\documentclass[12pt]{report}
\usepackage{amsmath,amssymb,amsfonts,amsthm} 
\usepackage{graphics}                 
\usepackage{diagrams}
\usepackage{graphicx}
\usepackage{subfigure}
\usepackage{geometry,setspace}
\usepackage{marvosym}
\newtheorem{theorem}{Theorem}[section]
\newtheorem{proposition}[theorem]{Proposition}
\newtheorem{corollary}[theorem]{Corollary}
\newtheorem{lemma}[theorem]{Lemma}
\newtheorem{remark}[theorem]{Remark}
\newtheorem{definition}[theorem]{Definition}
\newtheorem{problem}[theorem]{Problem}
\newtheorem{ex}[theorem]{Example}
\newtheorem*{mainthm}{Theorem~\ref{main}}
\newtheorem*{maincor}{Corollary~\ref{mc}}

\newcommand{\R}{{\mathbb {R}}}
\newcommand{\z}{{\mathbb {Z}}}

\newcommand{\Mod} {\mathrm{Mod}}
\newcommand{\h}{{\mathcal{H}}}
\newcommand{\J}{{\mathcal{J}}}
\newcommand{\I}{{\mathcal{I}}}
\newcommand{\Lie}{{\mathcal{L}}}
\newcommand{\Hom}{{\mathrm{Hom}}}

\newcommand{\im}{{\mathrm{im}}}
\newcommand{\p }{{\mathcal{P}}}
\newarrow{Mapsto}{|}{-}{-}{-}{>}
\newarrow{Bothto}{<}{-}{-}{-}{>}
\newarrow{Dashto}{}{dash}{}{dash}{>}

\begin{document}

\def\Xint#1{\mathchoice
{\XXint\displaystyle\textstyle{#1}}%
{\XXint\textstyle\scriptstyle{#1}}%
{\XXint\scriptstyle\scriptscriptstyle{#1}}%
{\XXint\scriptscriptstyle\scriptscriptstyle{#1}}%
\!\int}
\def\XXint#1#2#3{{\setbox0=\hbox{$#1{#2#3}{\int}$}
\vcenter{\hbox{$#2#3$}}\kern-.5\wd0}}
\def\ddashint{\Xint=}
\def\dashint{\Xint-}

\pagenumbering{roman}
\begin{spacing}{1.66}
\begin{center}
\thispagestyle{empty}
RICE UNIVERSITY \\
 
{\bfseries Surface Homeomorphisms That Do Not Extend to Any Handlebody and the Johnson Filtration} \\
 
by \\
 
{\bfseries Jamie Bradley Jorgensen} \\
 
\begin{singlespace}
A THESIS SUBMITTED \\
IN PARTIAL FULFILLMENT OF THE \\
REQUIREMENTS FOR THE DEGREE \\
\end{singlespace}
 
{\bfseries Doctor of Philosophy} \\

\begin{singlespace}
\begin{tabular}{rl}
\rule{13em}{0em} & \rule{18em}{0em} \\
& {\scshape Approved, Thesis Committee}: \\
& \\
& \\
& \underline{\hspace{20em}} \\
& Tim Cochran, Professor, Chair\\
& Mathematics \\
& \\
& \\
& \underline{\hspace{20em}} \\
& John Hempel, Milton B. Porter Professor\\
& Mathematics \\
& \\
& \\
& \underline{\hspace{20em}} \\
& Ian M. Duck, Professor \\
& Physics and Astronomy \\
& \\
& \\
& \\
 
\end{tabular}
\end{singlespace}
 
{\scshape Houston, Texas} \\
 
{\scshape April 2008} \\
\end{center}
\newpage

\begin{center}
\addcontentsline{toc}{chapter}{\numberline {}Abstract}
\thispagestyle{empty}
ABSTRACT
\bigskip
 
Surface Homeomorphisms That Do Not Extend to Any Handlebody and the Johnson Filtration
\bigskip
 
by
\bigskip
 
Jamie Bradley Jorgensen
\end{center}
We prove the existence of homeomorphisms of a closed, orientable surface of genus 3 or greater that do not extend to any handlebody bounded by the surface.  We show that such homeomorphisms exist arbitrarily deep in the Johnson filtration of the mapping class group.  The second and third terms of the Johnson filtration are the well-known Torelli group and Johnson subgroup, respectively.  
\newpage
\begin{center}
\addcontentsline{toc}{chapter}{\numberline {}Acknowledgments}
\thispagestyle{empty}
ACKNOWLEDGMENTS
\end{center}
\bigskip
In order of importance, these are the people to whom I dedicate this thesis:

1) Arlynda,

\qquad Who is everything to me;

\smallskip

2) My Children,

\qquad Who remind me what is important;

\smallskip
   
3) My Advisor, Tim Cochran,

\qquad For help, patience, and faith in my abilities;

\smallskip

4) My friends at Rice,

\qquad For good times and new perspectives;

\smallskip

5) Everyone who loves ${\mathcal M}$apping class groups.
\vfill
\hfill\Bicycle

\newpage
{\makeatletter
\let\ps@plain=\ps@empty
\addcontentsline{toc}{chapter}{\numberline {}Contents}
\tableofcontents}
\thispagestyle{myheadings}

\pagestyle{myheadings}

\begin{chapter}{Introduction}
\pagenumbering{arabic}
\setcounter{page}{1}

Surfaces, or two-dimensional manifolds, are among the most fundamental objects of study in mathematics.  
Surfaces are most easily pictured as subsets of $\R^3$.
Perhaps the first to consider \emph{abstract} surfaces was Riemann in his study of 
holomorphic functions.  
A homeomorphism of a surface is a continuous function from the surface to itself, having a continuous inverse.  Homeomorphisms can be thought of as the ``topological symmetries'' of a surface.  The study of surfaces has proven to be of interest in many areas of mathematics, such as low-dimensional topology and knot theory, algebraic geometry, complex analysis and differential geometry.  In fact, it is the rich interaction between these areas of mathematics that has made the study of surfaces so fruitful.  For instance, many topological results are proven by imposing a hyperbolic metric on a given surface.

The mapping class group of a surface is an object that algebraically ``encodes'' the topological symmetries of the surface.  A mapping class is an isotopy class of homeomorphisms, that is, a set of homeomorphisms that can be continuously deformed into one another.  In most instances when a surface is being studied---in any area of mathematics---there is a mapping class group lurking somewhere in the background.  

Many results in mapping class groups have some sort of algebraic hypothesis and a geometric conclusion.  For instance, the Nielsen realization theorem (which was proved by Kerckhoff, \cite{kerck}) says that finite subgroups of mapping class groups (satisfying some hypotheses) can be realized as \emph{isometries} of the surface in some hyperbolic metric.  Such results can be thought of as saying that, given certain algebraic constraints on mapping classes, some associated geometric behavior is nice.  

On the other hand, there are instances of mapping classes satisfying very strict algebraic constraints, but that misbehave geometrically.  
The result of this thesis can be thought of as an example of this latter phenomenon.  Namely, we show that there are homeomorphisms lying arbitrarily deep in the Johnson filtration of the mapping class group (this is our algebraic constraint) that do not extend to any handlebody (this is the geometric misbehavior).

\section{Preliminaries}

Let $\h$ be a handlebody of genus $g$, by which we mean a (closed) 3-ball with $g$ 1-handles added in such a manner that $\h$ is orientable.  The homeomorphism type of $\h$ depends only on $g$.  Let $\Sigma$ be the boundary of $\h$.  Then $\Sigma$ is an orientable, closed surface of genus $g$.  Suppose that $f:\Sigma\to\Sigma$ is a homeomorphism.  Then $f$ is said to \emph{extend} to $\h$ if there is a homeomorphism $F:\h\to\h$ that restricts to $f$ on the boundary.

The question of a homeomorphism of a surface extending to a handlebody is an important one.  For instance, suppose $M=\h_1\cup_\Sigma\h_2$ is a Heegaard splitting of a 3-manifold $M$.  Perturb $M$ to get a new 3-manifold $M'$ by cutting open along the surface $\Sigma$ and re-gluing via some homeomorphism $f:\Sigma\to\Sigma$.  If $f$ extends to either $\h_1$ or $\h_2$, then $M'$ is homeomorphic to $M$.

\begin{definition}\label{some}
A homeomorphism $f$ is said to extend to \emph{some} handlebody if there is a homeomorphism $h:\Sigma\to\partial\h$ (thinking now of $\Sigma$ as an abstract surface not associated with $\h$) such that $h\circ f\circ h^{-1}:\partial\h\to\partial\h$ extends to $\h$.
\end{definition}

By $j'$ we denote the inclusion $j':\partial\h\hookrightarrow\h$.  Given a homeomorphism $h:\Sigma\to\partial\h$, we will use the notation $j:=j'\circ h:\Sigma\hookrightarrow\h$.  Also, we will generally use the phrase ``let $j:\Sigma\hookrightarrow\h$ be a handlebody bounded by $\Sigma$'' to mean ``let $h:\Sigma\to\partial\h$ be a homeomorphism with $j=j'\circ h$.''

\section{Introductory examples}\label{sec:introexamples}

Having given Definition~\ref{some} it is often convenient to build handlebodies directly starting with the surface $\Sigma$ itself, as in the following example.

\begin{ex}\label{Dehntwist}
Let $u_1,\ldots,u_m$ be any set of pairwise-disjoint, simple closed curves on $\Sigma$.  Then any homeomorphism $f$ which is a composition of Dehn twists and inverse Dehn twists about the $u_i$ extends to some handlebody.  This can be seen by building the handlebody: Thicken $\Sigma$ by crossing with the interval $[0,1]$.  Attach 2-handles along each $u_i\times\{0\}$.  Attach more 2-handles if necessary to leave $\Sigma\times\{1\}$ and 2-spheres as the only boundary components.  Finally, cap off each such 2-sphere with a 3-ball.  The homeomorphism $f$ extends to the handlebody thus obtained.
\end{ex}

It is not difficult, at least in the case of genus one, to find homeomorphisms that do not extend to any handlebody, as demonstrated in the following example.

\begin{ex}\label{eigen}
Let $T$ be a torus and $f:T\to T$ a homeomorphism that extends to a
homeomorphism $F$  of some solid torus $\h$.  Let $\alpha$ be a curve
on $T$ which bounds a disk $\Delta$ in $\h$.  Then $\Delta$ must be carried to
a disk by $F$.  Thus, $[\alpha]$ and $f_*[\alpha]\in H_1(T,\R)$ 
are in the kernel of the inclusion induced map $i_*:H_1(T,\R) \to
H_1(\h,\R)$.  $i_*$ has one dimensional kernel so $[\alpha]$ and
$f_*[\alpha]$ must be linearly dependent.  Thus, $[\alpha]$ is an
eigenvector of $f_*$.  Since $f$ is a homeomorphism the eigenvalue corresponding to $[\alpha]$ must be $\pm 1$.  It follows that, for example, Anosov mapping classes of the torus do not extend to any handlebody.
\end{ex}

Some applications of homeomorphisms that do not extend to any handlebody are given by the following theorem found in \cite{CG}.

\begin{theorem}[Casson-Gordon]\label{ribbon}
A fibered knot in a homology 3-sphere is homotopically ribbon if and only if its closed monodromy extends to a handlebody.
\end{theorem}

For instance in \cite{B}, Bonahon has used Theorem~\ref{ribbon} to find an infinite family of knots, none of which is ribbon, but each of which is algebraically slice.  

\begin{ex}\label{slice}
Conversely, Theorem~\ref{ribbon} says that since the trefoil is
not slice (since it has signature -2, for instance) its closed monodromy $f:\Sigma\to\Sigma$ does not extend to any handlebody.
Of course we can also see that $f$ does not extend to any handlebody from its action on $H_1$, as in Example~\ref{eigen}.  By appropriately choosing a basis for $H_1(\Sigma)$, $f_*$ may be represented by the matrix
\[\left( \begin{array}{cc}
0 & -1 \\
1 & 1 \\
\end{array} \right)\]

(see, for example, \cite{R}).  Neither $\pm 1$ is an eigenvalue of this matrix.  Note, however, that $f$ is not an Anosov homeomorphism, as $f^6$ is the identity.

\end{ex}

Similarly, it is easy to find higher genus examples of homeomorphisms that do not 
extend to any handlebody by their action on $H_1$, as in the following example.

\begin{ex}\label{general}
As a generalization of the technique used in Example~\ref{eigen}, if a 
homeomorphism $f:\Sigma\to\Sigma$ extends to some
handlebody, then $\ker j_*$ is an invariant, $g$-dimensional subspace of $f_*:H_1(\Sigma,\R)\to H_1(\Sigma,\R)$.
Suppose $\Sigma$ has genus 2.  Fix a symplectic basis for $H_1(\Sigma)$.  Let $f$ be a homeomorphism of $\Sigma$
 whose induced transformation on $H_1(\Sigma)$ is given, in the chosen basis, by the following matrix:
\[
M=
\left(
\begin{array}{cccc}
0 & 0 & 0 & -1\\
1 & 0 & 0 & 0\\
0 & 1 & 0 & 0\\
0 & 0 & 1 & 0\\
\end{array}
\right)
\]
There are a couple of things to be aware of here.  First, if such an $f$ exists, its mapping class is by no means unique.  Second, since $M$ is a symplectic matrix, such an $f$ does exist (see, for example, \cite{primer}).  
We claim that such an $f$ does not extend to any handlebody.
It suffices to check that $M:\R^4\to\R^4$ has no invariant 2-dimensional subspaces.  First note that if $V\subset\R^4$ is an invariant subspace with basis $\{v_1,v_2\}$, then ${\mathrm{span}}_{\mathbb C}\{v_1,v_2\}$ is an invariant subspace of $M:{\mathbb C}^4\to{\mathbb C}^4$.  The 2-dimensional invariant subspaces of $M:{\mathbb C}^4\to{\mathbb C}^4$ can be found (using the techniques of \cite{subspaces}, for instance).  It is straightforward to check that none of these subspaces has a real basis.
\end{ex}

Theorem~\ref{slice} also provides examples of homeomorphisms that do not extend to any handlebody in arbitrary genus.  For instance, one could take the connect sum of trefoils for such an example.

\section{Model for $\Sigma$}\label{sec:model}
In this section we define a model that we will be using for the surface $\Sigma$.  The notational conventions defined in this section will be observed throughout this thesis.  However, in Section~\ref{robusthomeos} we will find it necessary to define a new model for $\Sigma$, but we will relate the new model to the model defined in this section.  We will find it necessary to work with both the genus $g$ surface $\Sigma$ and a genus $g$ subsurface $\Sigma_1$, with one boundary component.

Let $P$ be a regular polygonal region with $4g$ sides.
Let $\Sigma$ be obtained from $P$ by identifying the edges according to the word $\tilde a_1\tilde b_1\tilde a_1^{-1}\tilde b_1^{-1}\cdots \tilde a_g\tilde b_g\tilde a_g^{-1}\tilde b_g^{-1}$.  Let $D$ be a small disk at the center of $P$, as in Figure~\ref{fig:polygon}.
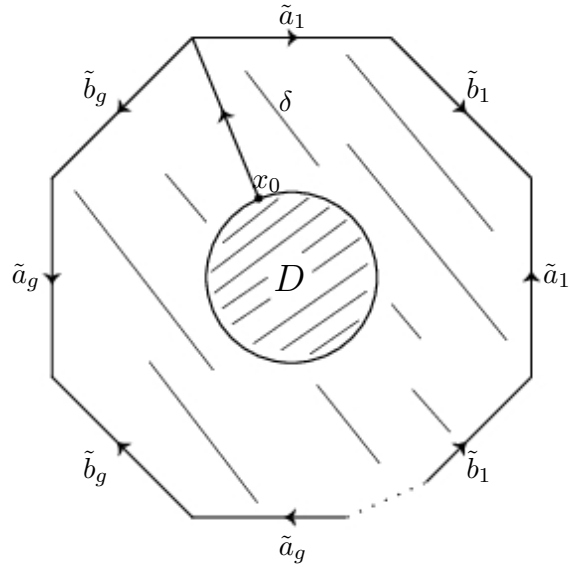
\begin{figure}[h!tp]
\centering
\begin{picture}(200,206)(0,0)
{\includegraphics[width=.49\textwidth]{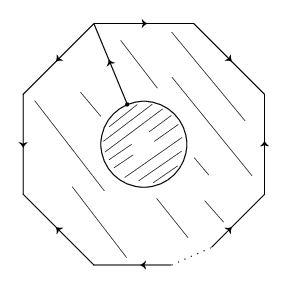}}
\put(-115,103){\large$D$}
\put(-123,142){\small$x_0$}
\put(-113,204){\small$\tilde a_1$}
\put(-42,176){\small$\tilde b_1$}
\put(-13,106){\small$\tilde a_1$}
\put(-42,32){\small$\tilde b_1$}
\put(-113,171){\small$\delta$}
\put(-113,3){\small$\tilde a_g$}
\put(-187,32){\small$\tilde b_g$}
\put(-214,106){\small$\tilde a_g$}
\put(-187,176){\small$\tilde b_g$}
\end{picture}
\caption{A model for the surface $\Sigma$ defining the curves $\tilde a_i,\tilde b_i$.}\label{fig:polygon}
\end{figure}
Let $\Sigma_1:=\Sigma-\mathrm{interior}(D)$, with inclusion denoted $i:\Sigma_1\hookrightarrow\Sigma$.
Let $x_0$ be a point on the boundary of $D$.  We will take all of our fundamental groups to be based at $x_0$.  For instance $\pi_1(\Sigma):=\pi_1(\Sigma,x_0)$ and $\pi_1(\Sigma_1):=\pi_1(\Sigma_1,x_0)$.
Let $\delta$ be a simple arc extending from $x_0\in\partial D$ to the vertex in $\partial P$, and let $a_i=\delta \tilde a_i\delta^{-1}$, $b_i=\delta \tilde b_i\delta^{-1}$.
Note that the isotopy classes $\{[a_i],[b_i]\in\pi_1(\Sigma_1)\}$ form a basis for the free group $\pi_1(\Sigma_1)$.  
Let $\alpha_i$ and $\beta_i$ be the images of $a_i$ and $b_i$, respectively in $H_1(\Sigma)$ (which we identify canonically with $H_1(\Sigma_1)$).  We orient $\Sigma$ by giving $P$ the ``into the page'' orientation.  Thus, the algebraic intersection number $\alpha_i\cdot\beta_i=1$, in accordance with the right-hand rule.  Let $\omega$ be the symplectic form on $H_1(\Sigma)$ given by algebraic intersection number.  Note that $\{\alpha_1,\ldots,\alpha_g,\beta_1,\ldots,\beta_g\}$ is a symplectic basis for $H_1(\Sigma)$.
\end{chapter}

\begin{chapter}{The Lie ring associated to a group}\label{sec:Lie}

\section{Basic definitions}

In this section we define the lower central series of a group, along with its associated Lie ring structure, and give some properties that we will need.  For a group $G$ and $x,y\in G$ we write the commutator of $x$ and $y$: $[x,y]=xyx^{-1}y^{-1}$.  For subgroups $H,K\subset G$ we let $[H,K]$ denote the subgroup of $G$ generated by all commutators $[x,y]$, with $x\in H$ and $y\in K$.  
\begin{definition}
The \emph{lower central series} of $G$ is the filtration 
\[
G=G_1\supset G_2\supset G_3\supset\cdots
\]
defined by $G_1:=G$ and $G_{k+1}:=[G,G_k]$.
\end{definition}

\begin{definition}
A \emph{Lie ring} $L$ is an abelian group (with group operation denoted by $+$), with a \emph{bracket operation} denoted $[\cdot,\cdot]$, satisfying:
\begin{itemize}
\item Bilinearity
\[
[x+y,z]=[x,z]+[y,z]\quad\textrm{and}\quad[x,y+z]=[x,y]+[x,z];
\]
\item Skew-commutativity
\[
[x,y]=-[y,x];\quad\textrm{and}
\]
\item Jacobi identity
\[
[x,[y,z]]+[y,[z,x]]+[z,[x,y]]=0
\]
\end{itemize}
for all $x,y,z\in L$.  $L$ is said to be \emph{graded} if $L$ decomposes (as an abelian group) as a direct sum
\[
{L}=\bigoplus_{i\in\z}{L}_i
\]
such that if $x\in{L}_i$ and $y\in{L}_j$ then $[x,y]\in{L}_{i+j}$.
\end{definition}

Of course a Lie ring is not a ring but a non-associative ring.  Lie rings also do not have a multiplicative identity.
We have used $[\cdot,\cdot]$ to denote both the commutator in a group and the bracket operation in a Lie ring.  We have done so because these two are intimately related in the context in which we will be using them.

For a group $G$ let
\[
\Lie(G):=\bigoplus_{i\ge1}\Lie_i(G),\quad\textrm{where}\quad\Lie_i(G):=\frac{G_i}{G_{i+1}}.
\]
$\Lie(G)$ has a natural graded Lie ring structure, with the bracket operation induced by the commutator in the following sense:  Suppose that $x\in G_i$, $y\in G_j$, then the commutator $[x,y]\in G_{i+j}$.  That is, $[G_i,G_j]\subset G_{i+j}$ (see, for instance, \cite{MKS} Theorem~5.3).  Thus, for
$\bar x\in\Lie_i(G)\subset\Lie(G)$ and $\bar y\in\Lie_j(G)\subset\Lie(G)$
we define
\[
[\bar x,\bar y]:=\overline{[x,y]}\in\Lie_{i+j}(G).
\]
\begin{proposition}
Extending the above definition of the bracket linearly to all of $\Lie(G)$ gives $\Lie(G)$ the structure of a Lie ring.
\end{proposition}
\begin{proof}
See \cite{MKS}, Theorem~5.3.
\end{proof}
The construction of the graded Lie ring structure on $\Lie(G)$ only depends on the fact that the lower central series is a \emph{central filtration} which we define now.
\begin{definition}
A sequence of subgroups $G_{\bar 1}$, $G_{\bar 2}$,$\ldots$ of a group $G$ is a \emph{central filtration} if
\begin{itemize}
\item $G_{\bar 1}=G$,
\item $G_{\overline{i+1}}\subset G_{\bar i}$, and
\item $[G_{\bar i},G_{\bar j}]\subset G_{\overline{i+j}}$
\end{itemize}
\end{definition}
The above proof has shown:
\begin{proposition}
Suppose that $G_{\bar i}$ is a central filtration of $G$.  Then
\[
\bar\Lie(G):=\frac{G_{\bar 1}}{G_{\bar 2}}\oplus\frac{G_{\bar 2}}{G_{\bar 3}}\oplus\cdots
\]
has a graded Lie ring structure with bracket operation induced by the commutator operation.
\end{proposition}

\section{Free Lie rings}\label{sec:freeLierings}

If $G$ is a free group, then $\Lie(G)$ is a \emph{free Lie ring}.  Free Lie rings may be defined by a universal mapping property, but it is probably easiest to think of them as Lie rings having no relations other than those imposed by the definition of a Lie ring.  Let $F=F^g$ be a free group on a set $X$ with $|X|=g$.  Then $\Lie(F)$ is a free Lie ring on the set $X$.  For each $k$, $\Lie_k(F)$ is a finite-rank, free abelian group.  The rank of $\Lie_k(F)$ is given by the number of \emph{basic commutators} of weight $k$ (see Definition~\ref{basiccomms} below).  In fact, the basic commutators of weight $k$ are elements of the group $F$ whose images in $\Lie_k(F)$ give a basis for $\Lie_k(F)$ (see, for example, \cite{hall} pp.\ 165-170).   
\begin{definition}\label{basiccomms}
The \emph{basic commutators} of weight 1 are the free generators of $F$, that is, the elements of $X$.  The basic commutators of weight $k$ are defined by first putting an order ``$<$'' on the basic commutators of weight less than $k$ such that $w<v$ if $w$ is a basic commutator of weight less than $v$.  The basic commutators of weight $k$ are the elements $[w_1,w_2]$, where $w_1$ and $w_2$ are basic commutators of weight $n_1$ and $n_2$ respectively such that $n_1+n_2=k$, and the following are satisfied 1) $w_1>w_2$ and 2) if $w_1=[v_1,v_2]$, then $v_2\le w_2$.
\end{definition}
There is a formula, attributed to Witt, for the number of basic commutators of weight $k$.
\begin{theorem}[Witt]\label{witt}
Let $C(k,g)$ be the number of basic commutators of weight $k$ (for a free group of rank $g$).  Then
\[
C(k,g)=\frac{1}{k}\sum_{d\vert k}\mu(d)g^{k/d}
\]
where $\mu(d)$ is the M\"obius function defined by
\[
\mu(d):=\left\{
\begin{array}{cc}
1&\textrm{if $d$ is square-free with an even number of prime factors (including $d=1$),}\\
-1&\textrm{if $d$ is square-free with an odd number of prime factors,}\hfill\\
0&\textrm{if $d$ is not square-free.}\hfill
\end{array}
\right.
\]
\end{theorem}
\begin{proof}
See \cite{hall} p.\ 170.
\end{proof}

\section{Lie ring homomorphisms and functorality}

If $I$ is a Lie subring of $L$ and $[I,L]\subset I$, then $I$ is said to be a an ideal of $L$.  Note that because of skew-commutativity we could replace the condition $[I,L]\subset I$ with $[L,I]\subset I$.  Thus left ideals, right ideals and two sided ideals all coincide for Lie rings.  If $I$ is an ideal of $L$, then we define the factor Lie ring $L/I$ via the equivalence relation
\[
x\sim y\quad\textrm{if and only if}\quad x-y\in I, 
\]
and the bracket operation is well defined on $L/I$ as a consequence of the condition $[I,L]\subset I$.  The following facts are fundamental (see, for example, \cite{Khuk}).
\begin{proposition}\label{prop:Liebasics}
Suppose $I$ is an ideal of $L$.
\begin{itemize}
\item If $\phi:L\to L'$ is a Lie ring homomorphism with kernel $I$, then $\phi(L)\cong L/I$.
\item If $J$ and $I$ are ideals of $L$ with $I\subset J$, then $J/I$ is an ideal of $L/I$ and the following diagram of canonical maps is commutative:
\[
\begin{diagram}
L  \\
\dTo & \rdTo \\
     &       & L/I\\
     &       & \dTo\\
L/J  & \rBothto^\cong&\frac{L/I}{J/I}
\end{diagram}
\]
\end{itemize}
\end{proposition}

\begin{lemma}
Given a group homomorphism $\phi:G\to G'$, there is an induced group homomorphism $\phi_*:\Lie_i(G)\to\Lie_i(G')$.
\end{lemma}
\begin{proof}
since $\phi(G_i)\subset G'_i$, we have a map $\phi:G_i\to G'_i$.  Since $\phi(G_{i+1})\subset G'_{i+1}$, it descends to a map $\phi_*:\Lie_i(G)\to\Lie_i(G')$.
\end{proof}

The homomorphism $\phi_*:\Lie_i(G)\to\Lie_i(G')$ extends uniquely to a Lie ring homomorphism $\phi_*:\Lie(G)\to\Lie(G')$.  Furthermore, if $N$ is the kernel of $\phi$, then let $N_{\bar i}$ be the filtration on $N$ defined by 
\[
N_{\bar i}:=N\cap G_i.
\]
Note that in general $N_{\bar i}\ne N_i$.  However, $N_{\bar i}$ is clearly a central filtration of $N$.
\begin{proposition}\label{prop:Liekernel}
Suppose that $\phi:G\twoheadrightarrow G'$ is an epimorphism.  Then
$\bar\Lie(N)$
is the kernel of the induced map $\phi_*:\Lie(G)\to\Lie(G')$.  In particular, $\Lie(G')\cong\Lie(G)/\bar\Lie(N)$.
\end{proposition}
\begin{proof}
$\bar\Lie(N)$ may be naturally identified with a Lie subring of $\Lie(G)$ since by definition
\[
\bar\Lie_i(N)=\frac{N_{\bar i}}{N_{\overline{i+1}}}=\frac{N_{\bar i}}{N\cap G_{i+1}}=\frac{N_{\bar i}}{G_{i+1}}\subset\Lie_i(G).
\]
Obviously $\ker(\phi:G_i\to G'_i)=N\cap G_i=N_{\bar i}$.  It follows that
$
\ker(\phi_*:\Lie_i(G)\to\Lie_i(G'))\subset\bar\Lie(N)
$.  Since $\phi$ is an epimorphism and since the terms of lower central series are \emph{verbal subgroups} (see \cite{MKS}) we have $\phi:G_i\twoheadrightarrow G'_i$ is an epimorphism.  Let $x\in G_i$.  Denote its image in $G_i/G_{i+1}$ by $\bar x$.  Suppose that $\phi_*(\bar x)=0$.  Thus, $\phi(x)\in G'_{i+1}$.  Hence, there is a $y\in G_{i+1}$, such that $\phi(y)=\phi(x)$.  Whence, $\phi(xy^{-1})=1$, which implies that $xy^{-1}$ is an element of $N_{\bar i}$, but since $\bar y=0$ in $G_i/G_{i+1}$, $\bar x$ is an element of $N_{\bar i}/N_{\overline{i+1}}$.
\end{proof}
If the condition that $\phi$ be an epimorphism is dropped from the hypothesis of Proposition~\ref{prop:Liekernel}, then the result fails to hold in general.  
\begin{ex}
Let $F$ be the free group generated by the symbols $x$ and $y$.  Let $\phi:F\to F$ be defined by
\[
\phi(x)=[[y,x],x]\quad\textrm{and}\quad\phi(y)=[[y,x],y].
\]
(Note that $[[y,x],x]$ and $[[y,x],y]$ are weight 3 basic commutators.)
One may check that though $\phi$ is injective, $\phi_*:\Lie(F)\to\Lie(F)$ is the zero map.
\end{ex}

\section{A result of Labute}
Let $G$ be a group given by a finite presentation.  That is, we are given a free group $F$ and a set of relators $\{r_1,\ldots,r_t\}\subset F$.  Let $R$ be the normal subgroup of $F$ normally generated by $\{r_1,\ldots,r_t\}$.  Thus, $G= F/R$.  Note that (assuming none of the relators is the identity) for each $i$ there is a largest $n(=n_i)$ such that $r_i\in F_n$ (since free groups are residually nilpotent).  We call $n$ the \emph{weight} of $r_i$.  Let $\rho_i$ be the image of $r_i$ in $\Lie_n(F)$ under the canonical map.  
Let $\mathfrak r$ be the Lie ring ideal of $\Lie(F)$ generated by $\{\rho_1,\ldots,\rho_t\}$.  In general $\Lie(F)/{\mathfrak r}\ne\Lie(G)$.  However, Labute proved:
\begin{theorem}[Labute]\label{thm:labute}
Under certain independence conditions on the relators $\{r_1,\ldots,r_t\}$, $\Lie(F)/{\mathfrak r}=\Lie(G)$.
\end{theorem}
We do not give the independence conditions in the hypothesis of Theorem~\ref{thm:labute} since Labute has checked that the conditions are satisfied in the two cases in which we are interested.  See \cite{Labute1} and \cite{Labute} for more details.

\section{Lie algebras}

Finally, we mention the definition of a Lie algebra.
\begin{definition}
A \emph{Lie algebra} over a field $\mathbb F$ is a Lie ring $A$ that is also a vector space over $\mathbb F$ and satisfies:
\begin{itemize}
\item Bilinearity over $\mathbb F$
\[
[rx+sy,z]=r[x,z]+s[y,z]\quad\textrm{and}\quad[x,ry+sz]=r[x,y]+s[x,z],
\]
\end{itemize}
for all $x,y,z\in A$ and all $r,s\in{\mathbb F}$.
\end{definition}

\begin{remark}\label{lem:liealg}
If $L$ is a Lie ring, then $L\otimes{\mathbb F}$ is a Lie algebra over $\mathbb F$ via:
\[
[\ell\otimes f,\ell'\otimes f']:=[\ell,\ell']\otimes ff'.
\]
\end{remark}

\end{chapter}

\begin{chapter}{The mapping class group and its filtrations}\label{sec:Jfilt}

This thesis aims to prove the existence of homeomorphisms that do not extend to any handlebody but whose action on the fundamental group of the surface is more subtle than in the examples of Section~\ref{sec:introexamples}.  We take the mapping class group of $\Sigma$ to be the group of orientation preserving homeomorphisms of $\Sigma$, up to isotopy.  We denote it by $\Mod(\Sigma)$.  For technical reasons we work with $\Mod(\Sigma,D)$, the group of orientation preserving homeomorphisms that fix $D$ pointwise, modulo isotopies that also fix $D$ pointwise.  However, note that whether or not a homeomorphism extends to any handlebody depends only on its mapping class in $\Mod(\Sigma)$, so we may work exclusively with $\Mod(\Sigma,D)$, and analogous results hold for $\Mod(\Sigma)$.

\section{The Johnson filtration}
Let
\[
N_k:=\frac{\pi_1(\Sigma)}{\pi_1(\Sigma)_k},\quad\textrm{and}\quad N_{1,k}:=\frac{\pi_1(\Sigma_1)}{\pi_1(\Sigma_1)_k}.
\]
To state the main problem of this thesis we give the following definition originally given by Johnson in \cite{J83}.

\begin{definition}
The Johnson filtration of $\Mod(\Sigma,D)$
\[
\cdots\subset\J(k)\subset\cdots\subset\J(3)\subset\J(2)\subset\J(1)=\Mod(\Sigma,D)
\]
is defined by 
\[
\J(k):=\{[f]\in\Mod(\Sigma,D)\vert f'_*:N_{1,k}\to N_{1,k} {\textrm{ is the identity map}}\},
\]
where $f'=f\vert_{\Sigma_1}$.
\end{definition}

Note that $\J(1)=\Mod(\Sigma,D)$, $\J(2)=\mathcal{I}$ is the Torelli group, and $\J(3)=\mathcal{K}$ is the Johnson subgroup.

\begin{definition}\label{defn:altfilt}
We define another filtration of $\Mod(\Sigma,D)$,
\[
\cdots\subset\J'(k)\subset\cdots\subset\J'(3)\subset\J'(2)\subset\J'(1)=\Mod(\Sigma,D),
\]
by
\[
\J'(k):=\{[f]\in\Mod(\Sigma,D)\vert f_*:N_k\to N_k {\textrm{ is the identity map}}\}.
\]
\end{definition}
It turns out that:
\begin{proposition}\label{prop:altfilt}
In general $\J(k)\subset\J'(k)$, but for $k\ge4$, $\J(k)\ne\J'(k)$.
\end{proposition}
\begin{proof}
We first show that $\J(k)\subset\J'(k)$.  Let $f:\Sigma\to\Sigma$ be a homeomorphism that fixes the disk $D$ pointwise and let $f':\Sigma_1\to\Sigma_1$ be its restriction.  Suppose that $[f]\in\J(k)$.  That is, $f'_*(y)y^{-1}\in \pi_1(\Sigma_1)_k$ for all $y\in\pi_1(\Sigma_1)$.  
We have the commutative diagram:
\[
\begin{diagram}
\pi_1(\Sigma_1)      & \rTo^{f'_*}  & \pi_1(\Sigma_1)            \\
\dTo_{i_*} &                  & \dTo_{i_*}             \\
\pi_1(\Sigma)     & \rTo^{f_*}   &\pi_1(\Sigma)
\end{diagram}
\]
Let $x\in\pi_1(\Sigma)$.  Select $x_1\in\pi_1(\Sigma_1)$ such that $i_*(x_1)=x$.  Then $f_*(x)=i_*\circ f'_*(x_1)$ and thus $f_*(x)x^{-1}=i_*\circ f'_*(x_1)i_*(x_1^{-1})=i_*(f'_*(x_1)x_1^{-1})$.  Since $f'_*(x_1)x_1^{-1}\in\pi_1(\Sigma_1)_k$, and since the terms of the lower central series are verbal subgroups $i_*(f'_*(x_1)x_1^{-1})\in\pi_1(\Sigma)$ so $[f]\in\J'(k)$.  For the second part we produce a homeomorphism that lies arbitrarily deep in the primed filtration, but is not in the fourth term of the Johnson filtration.    

Let $c$ be a simple closed curve in the interior of $\Sigma_1$, parallel to its boundary.
Let $t$ be the Dehn twist about $c$.  Note that $t_*$ acts as the identity on $\pi_1(\Sigma)$, as the curve $c$ is null-homotopic in $\Sigma$.  However, $[t]\ne1$ in the group $\Mod(\Sigma,D)$ as we show presently.  Consider the element $t_*(b_g)b_g^{-1}=[[a_1,b_1][a_2,b_2]\cdots[a_g,b_g],b_g]\in\pi_1(\Sigma_1)$.  Modulo $\pi_1(\Sigma_1)_3$, we have $[a_1,b_1][a_2,b_2]\cdots[a_g,b_g]\equiv\sum_i[\alpha_i,\beta_i]$.  Thus, modulo $\pi_1(\Sigma_1)_4$, $[[a_1,b_1][a_2,b_2]\cdots[a_g,b_g],b_g]\equiv\sum_i[[\alpha_i,\beta_i],\beta_g]$, which is non-zero in $\pi_1(\Sigma_1)_3/\pi_1(\Sigma_1)_4$, as the set of terms of this sum is linearly independent. (The terms $[[\alpha_i,\beta_i],\beta_g]$ are the images of basic commutators.  See Section~\ref{sec:freeLierings}.)  Therefore, $[t]\notin\J(4)$.
\end{proof}

In the literature a Johnson filtration of the group $\Mod(\Sigma)$ is also defined.
\begin{definition}
The Johnson filtration of the group $\Mod(\Sigma)$
\[
\cdots\subset\J_0(k)\subset\cdots\subset\J_0(3)\subset\J_0(2)\subset\J_0(1)=\Mod(\Sigma)
\]
is defined by setting $\J_0(k)$ to be the image of $\J(k)$ under the natural quotient $\Mod(\Sigma,D)\twoheadrightarrow\Mod(\Sigma)$.
\end{definition}
See, for example, \cite{survey}.  However, this filtration does not play a significant role in this work.

Now we may state the main problem that this thesis addresses:
\begin{problem}
Are there homeomorphisms $[f]\in\J(k)$ that do not extend to any handlebody?
\end{problem}

Note that the answer to the above question may in general depend on the positive integer $k$ and on the genus of the surface in question.

A partial result was given by Casson in 1979 and by Johannson and Johnson in 1980.  Neither of these works was published.  However, Leininger and Reid  \cite{LR} published an outline of Johannson and Johnson's proof.    The result is as follows.

\begin{theorem}
If the genus of $\Sigma$ is 2 or larger, then there exist homeomorphisms $[f]\in\I$ that do not extend to any handlebody and for
every odd integer $n$, $f^n$ does not extend to any handlebody.
\end{theorem}
See \cite{LR} for details.

\section{Statement of results}

The main result of this thesis is:

\begin{mainthm}
Given an integer $k\ge2$, there are homeomorphisms $[f]\in\J(k)-\J(k+1)$ that do not extend to any handlebody, provided that the genus of the surface is at least 7.  
Furthermore, $f^n$ does not extend to any handlebody for any integer $n\ne0$.
For $k\ge3$ the result also holds for genera 5 and 6.  For $k\ge4$ the result also holds for genus 4 and if $k\ge6$, then the result holds for genus 3.
\end{mainthm}

Very similar results have been obtained independently by Richard Hain \cite{Hain} using different methods.  We have immediately:

\begin{maincor}
If $\Sigma$ is a closed, orientable surface of genus 3 or greater, then there are homeomorphisms of $\Sigma$ that do not extend to any handlebody, lying arbitrarily deep in the Johnson filtration.
\end{maincor}

The primary tool that we use in proving Theorem~\ref{main} is the $k^{\textrm{th}}$ Johnson homomorphism\footnote{Technically the homomorphism denoted by $\tau_{1,k}$ is called the Johnson homomorphism.  However, $\tau_{k}$ and $\tau_{1,k}$ are so closely related that we will  refer to each of them as the Johnson homomorphisms.} $\tau_k$, which is defined in Section~\ref{johnsonhom}.
The proof of Theorem~\ref{main} comes in two main parts.  First, we show that if $f$ is a homeomorphism that extends to a given handlebody $j:\Sigma\hookrightarrow\h$ bounded by $\Sigma$ then $j_*\circ\tau_k[f]=0$.  Chapter~\ref{johnsonhom} and Section~\ref{sec:behavior} are devoted to proving this first part.  Second, we show that there exist homeomorphisms $f$ with $j_*\circ\tau_k[f]\ne0$ for all handlebodies $j:\Sigma\hookrightarrow\h$ bounded by $\Sigma$.  We call such homeomorphisms \emph{robust}.  Proof of this second part is the subject of the remainder of Chapter~\ref{handlebodiesandtau}.
\end{chapter}

\begin{chapter}{The Johnson and Morita-Heap homomorphisms}\label{johnsonhom}

\section{The Johnson homomorphisms}\label{sec:johnsonhom}

Dennis Johnson defined a homomorphism in \cite{J80} from the Torelli group
$\I$ onto a 
free abelian group.  In his survey \cite{J83}, Johnson defines a 
sequence of homomorphisms, $\tau_{1,k}$, on the terms of the Johnson filtration.   In this section we first introduce a homomorphism $\tau_k$ closely related to $\tau_{1,k}$.  Next we will define $\tau_{1,k}$ itself. We will use $\tau_k$ to detect the existence of homeomorphisms that do not extend to any handlebody.   For technical reasons we are forced to use both of these homomorphisms.  In the next section we elaborate the relationship between the two.  

Let $\Lie_{(1)}:=\Lie(\Sigma_1)$ and $\Lie_{1,k}:=\Lie_k(\Sigma_1)$.  Since $\pi_1(\Sigma_1)$ is free, $\Lie_{(1)}$ is a free Lie ring generated by $\Lie_{1,1}=H_1(\Sigma_1)$.  We also define $\Lie:=\Lie(\Sigma)$ and $\Lie_k:=\Lie_k(\Sigma)$.  
Let $\mathfrak I\subset\Lie_{(1)}$ be the Lie ring ideal generated by the \emph{symplectic class} $\Omega:=[\alpha_1,\beta_1]+[\alpha_2,\beta_2]+\cdots+[\alpha_g,\beta_g]$.  We have:
\begin{lemma}\label{labute1}
\[
\Lie=\frac{\Lie_{(1)}}{\mathfrak I}
\]
and the canonical epimorphism $\Lie_{(1)}\to\Lie_{(1)}/{\mathfrak I}=\Lie$ is the inclusion induced map. 
\end{lemma}
\begin{proof}
This is an application of Theorem~\ref{thm:labute}.  The hypothesis is checked in \cite{Labute1}.
\end{proof} 

For each $[f]\in\J(k)$ there is a homomorphism from $H_1(\Sigma)$ to $\Lie_k$ given by $[x]\mapsto[f_*(x)x^{-1}]$ (where $x$ is an element of $\pi_1(\Sigma)$).  Note that $f_*(x)x^{-1}\in\pi_1(\Sigma)_k$ since $[f]\in\J(k)\subset\J'(k)$.
We may define a homomorphism $\sigma:\J(k)\to \Hom(H_1(\Sigma),\Lie_k)$ by
\[
[f] \rMapsto ([x]\mapsto[f_*(x)x^{-1}]).
\]
This is almost the homomorphism that we want.  The reader is referred to \cite{J80} and \cite{J83} for proof that the maps $H_1(\Sigma)\to\Lie_k$ and $\sigma:\J(k)\to \Hom(H_1(\Sigma),\Lie_k)$ are well defined and in fact homomorphisms.  To get the homomorphism $\tau_k$ we note that the symplectic form $\omega$ on $H_1(\Sigma)$ gives a canonical homomorphism $\eta:H_1(\Sigma)\otimes\Lie_k\to\Hom(H_1(\Sigma),\Lie_k)$, defined by
\[
\eta(h\otimes l):=([x] \mapsto \omega(h,[x])\cdot l).
\]
\begin{lemma}
$\eta$ is an isomorphism. 
\end{lemma}
\begin{proof}
First $\eta$ is onto.  To show this, consider an arbitrary homomorphism $\kappa\in\Hom(H_1(\Sigma),\Lie_k)$.  $\kappa$ is uniquely determined by its values on a basis.  Let $\ell_i:=\kappa(\alpha_i)$, $\ell_i':=\kappa(\beta_i)$.  If we let
\[
x:=\sum_i\left(\alpha_i\otimes \ell_{i}'-\beta_i\otimes \ell_{i}\right),
\]
then it is clear that $\eta(x)=\kappa$.  We show that $\eta$ is one-to-one.  Let 
\[
x=\sum_i\left(\alpha_i\otimes l_{i}+\beta_i\otimes l_{i}'\right)  
\]
be an arbitrary element of $H_1(\Sigma)\otimes\Lie_k$, such that $\eta(x)=0$.  We show that $x=0$.  We compute:
\[
0=\eta(x)(\alpha_k)=\sum_i\left(\omega(\alpha_i,\alpha_k)l_{i}+\omega(\beta_i,\alpha_k)l_i'\right)=-l_k'.
\]
Thus $l_k'$ is zero for all $k$ and similarly for $l_k$.
\end{proof}

Now we are ready for our homomorphism.
\begin{definition}
The Johnson homomorphism $\tau_k:\J(k)\to H_1(\Sigma)\otimes\Lie_k$ is defined to be $\tau_k:=\eta^{-1}\circ\sigma$.
\end{definition}

The $k^{\textrm{th}}$ ($k\ge2$) Johnson 
homomorphism $\tau_{1,k}:\J(k)\to H_1(\Sigma_1)\otimes\Lie_{1,k}$ is defined very similarly to above.  In particular we define $\sigma':\J(k)\to \Hom(H_1(\Sigma_1),\Lie_{1,k})$ by
\[
[f] \rMapsto ([x]\mapsto[f'_*(x)x^{-1}]),\quad\textrm{(where }f'=f\vert\Sigma_1).
\]
There is a canonical isomorphism
\[
\eta':H_1(\Sigma_1)\otimes\Lie_{1,k}\to\Hom(H_1(\Sigma_1),\Lie_{1,k}).
\]
\begin{definition}
The Johnson homomorphism $\tau_{1,k}:\J(k)\to H_1(\Sigma_1)\otimes\Lie_{1,k}$ is defined to be $\tau_{1,k}:={\eta'}^{-1}\circ\sigma$.
\end{definition}

See \cite{J80} and \cite{J83} for details.

\section{The Morita-Heap homomorphisms}\label{moritahom}
Ultimately we will use $\tau_k$ to detect the existence of homeomorphisms that do not extend to any handlebody.  In order to do so we will need a homomorphism $\bar\tau_{1,k}$ originally defined by Morita, and a closely related homomorphism $\bar\tau_k$. 

In \cite{abelian} Morita defines a ``refinement'' of the Johnson homomorphism $\bar \tau_{1,k}:\J(k)\to H_3(N_{1,k})$, such that:
\begin{lemma}[Morita and Heap]\label{heap}
the diagram
$$
\begin{diagram}
    &                                & H_3(N_{1,k} )            \\
    &\ruTo^{\bar\tau_{1,k} } & \dTo_{d^2_1}             \\
\J(k) & \rTo^{\tau_{1,k}}          &H_1(\Sigma_1) \otimes \Lie_{1,k}
\end{diagram}
$$
commutes.  
\end{lemma}
We define $d^2_1:H_3(N_{1,k})\to H_1(\Sigma_1) \otimes \Lie_{1,k}$ below.
We call $\bar\tau_{1,k}$ a refinement of $\tau_{1,k}$ because the latter factors through the former.
We will not give Morita's original definition of $\bar\tau_{1,k}$ here; rather, we give an alternative definition shown to be equivalent by Heap in \cite{H}.  See \cite{abelian} and \cite{H} for proof of lemma~\ref{heap}.  We let $T_{f,1}$ be the mapping torus of $f'=f\vert_{\Sigma_1}$.  That is, $T_{f,1}=\Sigma_1\times [0,1]/(x,0)\sim(f'(x),1)$.  The boundary $\partial T_{f,1}$ is the torus $\partial\Sigma_1\times S^1$.  By gluing in a solid torus, with meridian along the circle $\{x_0\}\times S^1$, we obtain a closed, oriented 3-manifold (that depends only on the mapping class $[f]$ up to homeomorphism) which we denote (following Heap) by $T_f^\gamma$.  Letting $\alpha_{g+i}:=\beta_i$ for $i=1,\ldots,g$, a presentation of $\pi_1(T_{f,1})$ is as follows:
\[
\pi_1(T_{f,1})=\langle\alpha_1,\ldots,\alpha_{2g},\gamma\,|\,[\alpha_1,\gamma]f_*(\alpha_1)\alpha_1^{-1},\ldots,[\alpha_{2g},\gamma]f_*(\alpha_{2g})\alpha_{2g}^{-1}\rangle,
\]
where $\gamma$ is the homotopy class of the curve $\{x_0\}\times S^1$, given an arbitrary orientation.  $T_f^\gamma$ is obtained from $T_{f,1}$ by a Dehn filling along $\gamma$.  There is a presentation of $\pi_1(T_f^\gamma)$ as follows:
\[
\pi_1(T_f^\gamma)=\langle\alpha_1,\ldots,\alpha_{2g}\,|\,f_*(\alpha_1)\alpha_1^{-1},\ldots,f_*(\alpha_{2g})\alpha_{2g}^{-1}\rangle.
\]
We define a continuous map $\phi_{f,k}^1:T_{f,1}\to K(N_{1,k},1)$, where $K(N_{1,k},1)$ is an Eilenberg-MacLane space for the group $N_{1,k}$, induced by the canonical epimorphism:
\begin{equation}\label{tf1}
\pi_1(T_{f,1})\twoheadrightarrow\frac{\pi_1(T_{f,1})}{\langle\gamma,\pi_1(T_{f,1})_k\rangle}\cong\frac{\pi_1(\Sigma_1)}{\pi_1(\Sigma_1)_k}=N_{1,k}.
\end{equation}
The isomorphism of (\ref{tf1}) and the existence of the map $\phi_{f,k}^1$ follows from Lemma~\ref{equivalent} below.  Also, since $T_f^\gamma$ is obtained from $T_{f,1}$ by attaching a 2-cell along $\gamma$ (and then attaching a 3-cell) and since $\gamma$ is in the kernel of the canonical epimorphism~\eqref{tf1}, it follows that $\phi_{f,k}^1$ extends to a continuous map $\phi_{f,k}^\gamma:T_f^\gamma\to K(N_{1,k},1)$, and $\phi_{f,k}^\gamma$ induces the canonical epimorphism
\[
\pi_1(T_f^\gamma)\twoheadrightarrow\frac{\pi_1(T_{f}^\gamma)}{\pi_1(T_{f}^\gamma)_k}\cong N_{1,k}.
\]
The following lemma is given in \cite{H}, where it is labeled Lemma~5.1.
\begin{lemma}[Heap]\label{equivalent}
The following are equivalent:
\begin{enumerate}
\item $f\in\J(k)$,
\item $\frac{\pi_1(T_{f,1})}{\langle\gamma,\pi_1(T_{f,1})_k\rangle}\cong N_{1,k}\cong\frac{\pi_1(T_f^\gamma)}{\pi_1(T_f^\gamma)_k}$, and
\item the continuous maps $\phi_{f,k}^1$ and $\phi_{f,k}^\gamma$ exist as defined.
\end{enumerate}
\end{lemma}

The proof of Lemma~\ref{equivalent} is straightforward and is given in \cite{H}.  We are now ready to define $\bar\tau_{1,k}:\J(k)\to H_3(N_{1,k})$.  Note that we can take $H_3(N_{1,k})=H_3(K(N_{1,k},1))$ and define
\[
\bar\tau_{1,k}([f])=\phi^\gamma_{f,k *}([T_f^\gamma]).
\]
That is, $[f]$ is taken to the image of the fundamental class of the closed 3-manifold $T_f^\gamma$ in $H_3(N_{1,k})$.  That $\bar\tau_{1,k}$ is well defined is immediate since if $f$ and $g$ are isotopic then $T_f^\gamma$ and $T_g^\gamma$ are homeomorphic, with the homeomorphism being canonical up to isotopy.  It is not immediate that $\bar\tau_{1,k}$ is a homomorphism.  One must show that $\bar\tau_{1,k}(f\circ g)=\bar\tau_{1,k}(f)+\bar\tau_{1,k}(g)$ for all $f,g\in\J(k)$.  The main thrust of the proof is to build a 4-manifold whose (oriented) boundary is $T_f^\gamma\coprod T_g^\gamma\coprod -T_{f\circ g}^\gamma$ and show that the maps $\phi_{f,k}^\gamma$, $\phi_{g,k}^\gamma$ and $\phi_{f\circ g,k}^\gamma$ extend over this manifold.  Details are provided in \cite{H}.

We now describe the map $d^2_1$. We have the central extension
\begin{equation}\label{firstextension}
0\to\Lie_{1,k}\to N_{1,k+1}\to N_{1,k}\to1.
\end{equation}
Associated to this extension is a Hochschild-Serre spectral sequence for the \emph{homology} of $N_{1,k}$, $(E^r_{p,q}, d^r_1)$ with $E^2_{p,q}=H_p(N_{1,k};H_q(\Lie_{1,k}))$.  Since the extension (\ref{firstextension}) is central, we have $E^2_{p,q}=H_p(N_{1,k})\otimes H_q(\Lie_{1,k})$. Finally, we have the differential $d^2_1:E^2_{3,0}=H_3(N_{1,k})\to E^2_{1,1}=H_1(N_{1,k})\otimes H_1(\Lie_{1,k})=H_1(\Sigma_1)\otimes\Lie_{1,k}$.  We have now defined all of the maps appearing in Lemma~\ref{heap}.

We define a homomorphism $\bar\tau_k:\J(k)\to H_3(N_k)$.  This definition is a slight modification of Heap's definition of $\bar\tau_{1,k}$.  $\bar\tau_{1,k}$ and $\bar\tau_k$ are closely related as will be shown in Section~\ref{propproof}.  For $[f]\in\J(k)$ let  $T_f$ be the mapping torus of $f$.  That is $T_f=\Sigma\times[0,1]/(x,0)\sim(f(x),1)$.  Note that $T_f$ is a closed oriented 3-manifold that depends only on $[f]$ up to homeomorphism.  Similar to above we define a continuous map $\phi_{f,k}:T_f\to K(N_k,1)$ induced by the canonical epimorphism
$$
\pi_1(T_f)\twoheadrightarrow\frac{\pi_1(T_f)}{\langle\gamma,\pi_1(T_f)_k\rangle}\cong N_{k}.
$$
Analogous to Lemma~\ref{equivalent} we have:
\begin{lemma}\label{equivalent2}
The following are equivalent:
\begin{enumerate}
\item $f\in\J'(k)$,
\item $\frac{\pi_1(T_{f})}{\langle\gamma,\pi_1(T_{f})_k\rangle}\cong N_k$, and
\item the continuous map $\phi_{f,k}$ exists as defined.
\end{enumerate}
\end{lemma}

We take $H_3(N_{k})=H_3(K(N_{k},1))$ and define
\[
\bar\tau_{k}([f])=\phi_{f,k *}([T_f^\gamma]).
\]
That is $[f]$ is taken to the image of the fundamental class of the closed 3-manifold $T_f$ in $H_3(N_{k})$.  
From Lemma~\ref{equivalent2} we could take $\bar\tau_k$ to be a homomorphism with domain the group $\J'(k)$; however, we will content ourselves to think of the domain of $\bar\tau_k$ as the subgroup $\J(k)\subset\J'(k)$.

There is a Hochschild-Serre spectral sequence associated to the central extension
\[
0\to\Lie_{k}\to N_{k+1}\to N_{k}\to1
\]
and this yields a differential $d^2:H_3(N_k)\to H_1(\Sigma) \otimes \Lie_{k}$.  Similar to lemma~\ref{heap} we have
\begin{proposition}\label{diagram}
The diagram
\[
\begin{diagram}
    &                                & H_3(N_{k} )            \\
    &\ruTo^{\bar\tau_{k} } & \dTo_{d^2}             \\
\J(k) & \rTo^{\tau_{k}}          &H_1(\Sigma) \otimes \Lie_{k}
\end{diagram}
\]
commutes.
\end{proposition}

Proof of proposition~\ref{diagram} is the content of the following section.

\section{Proof of proposition~\ref{diagram} }\label{propproof}

Proposition~\ref{diagram} follows from Lemma~\ref{heap} and Lemmas~\ref{lemma1}, \ref{lemma2} and \ref{lemma3} below.

\begin{lemma}\label{lemma1}
The diagram
\[
\begin{diagram}
    &                                &H_1(\Sigma_1) \otimes \Lie_{1,k} \\
    &\ruTo^{\tau_{1,k} } & \dTo_{i_*}             \\
\J(k) & \rTo^{\tau_{k}}          &H_1(\Sigma) \otimes \Lie_{k}
\end{diagram}
\]
commutes.
\end{lemma}

\begin{proof}
We compute in the symplectic basis $\{\alpha_1,\ldots,\alpha_g,\beta_1,\ldots, \beta_g\}$ for $H_1(\Sigma_1)$.  We have under the map $\tau_{1,k}$:
$$
[f]\rMapsto([x]\mapsto[f'_*(x)x^{-1}])\rMapsto\sum_i\alpha_i\otimes[f'(b_i)b_i^{-1}]-\sum_i\beta_i\otimes[f'(a_i)a_i^{-1}].
$$
But the curves $f'(a_i)$ and $f'(b_i)$ are exactly the curves $f(a_i)$ and $f(b_i)$ respectively.  So under the inclusion map $i_*$, the element $\sum_i\alpha_i\otimes[f'(b_i)b_i^{-1}]-\sum_i\beta_i\otimes[f'(a_i)a_i^{-1}]$ gets sent to
$$
\sum_i\alpha_i\otimes[f(b_i)b_i^{-1}]-\sum_i\beta_i\otimes[f(a_i)a_i^{-1}]=\tau_k[f].
$$
\end{proof}

Of course inclusion $i:\Sigma_1\hookrightarrow\Sigma$ induces a map $i_*:H_3(N_{1,k})\to H_3(N_k)$.  We have:

\begin{lemma}\label{lemma2}
The diagram
\[
\begin{diagram}
    &                                &H_3(N_{1,k}) \\
    &\ruTo^{\bar\tau_{1,k} } & \dTo_{i_*}    \\
\J(k) & \rTo^{\bar\tau_{k}}      &H_3(N_{k})
\end{diagram}
\]
commutes.
\end{lemma}

\begin{proof}
The proof is by construction of a cobordism between $T_{f}^\gamma$ and $T_{f}$ over the group $N_k$.  Construct the cobordism as follows:  Let $X=T_{f}\times[0,1]$, which has two copies of $T_{f}$ as its boundary.  Now we glue $X$ and a thickened $S^3$, (which decomposes as the union of two solid tori thickened) $S^3\times[0,1]=(B^2\times S^1\cup S^1\times B^2)\times[0,1]$, along pieces of their boundaries.
\begin{figure}[h!tp]
\centering
\begin{picture}(195,190)(0,0)
\subfigure[The manifold $X$.]
{\includegraphics[width=.48\textwidth]{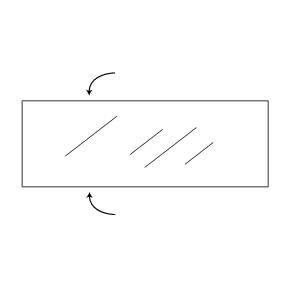}}
\put(-120,51){\small$T_f\times\{0\}$}
\put(-120,155){\small$T_f\times\{1\}$}
\end{picture}
\hspace{.02\textwidth}
\begin{picture}(195,190)(0,0)
\subfigure[{Decomposition of $S^3\times [0,1]$.}]
{\includegraphics[width=.48\textwidth]{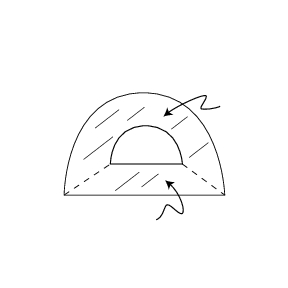}}
\put(-136,47){\small$B^2\times S^1$}
\put(-48,131){\small$S^1\times B^2$}
\end{picture}
\vspace{1cm}
\caption{Schematic of the spaces $X$ and $S^3\times[0,1]$.}
\end{figure}
Identify in the top copy $T_{f}\times\{1\}$ the solid torus $(D\times S^1)\times\{1\}\subset T_{f}\times\{1\}$ with the solid torus $(B^2\times S^1)\times\{0\}\subset S^3\times\{0\}$ in the bottom copy of $S^3$, by gluing the longitude $\{x_0\}\times S^1$ along the longitude $\{b_0\}\times S^1$ (for some $b_0\in\partial B^2$) .
\begin{figure}[h!tp]
\centering
\begin{picture}(435,200)(0,0)
\includegraphics[width=.98\textwidth]{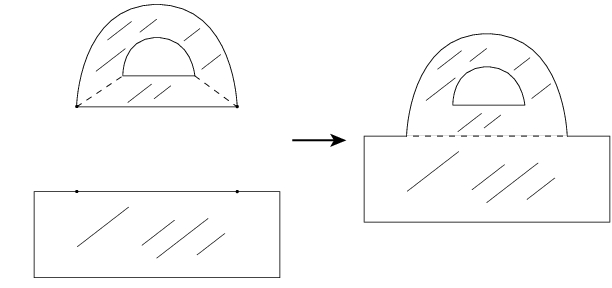}
\put(-378,70){$\overbrace{\qquad\qquad\qquad\qquad\quad \phantom{a}}$}
\put(-353,85){\footnotesize$(D\times S^1)\times\{1\}$}
\end{picture}
\caption{Obtaining $W$ by gluing $X$ and $S^3\times[0,1]$ together.}
\end{figure}
Call the resulting 4-manifold $W$.  $W$ has 3 boundary components homeomorphic to $T_{f}$, $T_{f}^\gamma$, and $S^3$ respectively.  The desired cobordism $\hat W$ is obtained by capping the $S^3$ boundary component off with a 4-ball.
\begin{figure}[h!tp]
\centering
\begin{picture}(195,190)(0,0)
\subfigure[Schematic of $\partial W$.]
{\includegraphics[width=.48\textwidth]{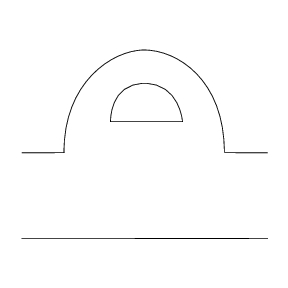}}
\put(-70,167){\tiny$\big(T_f-(D\times S^1)\big)\cup(S^1\times B^2)\cong T_f^\gamma$}
\put(-115,110){\small$S^3\times\{1\}$}
\put(-115,42){\small$T_f\times\{0\}$}
\end{picture}
\hspace{.04\textwidth}
\begin{picture}(195,190)(0,0)
\subfigure[Schematic of $\hat W$.]
{\includegraphics[width=.48\textwidth]{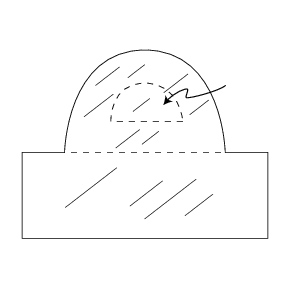}}
\put(-37,147){$B^4$}
\end{picture}
\vspace{0.5cm}
\caption{The manifolds $\partial W$ and $\hat W$.}
\end{figure}Now one needs only to check that there is a map $\Phi:\pi_1(\hat W)\to N_{k}$, and a commutative diagram
\[
\begin{diagram}
\pi_1(T_{f}^\gamma)&                          &           \\
\dTo^{i_{1*}} &\rdTo^{\phi^\gamma_{f,k}} &            \\
\pi_1(\hat W)   & \rDashto^{\Phi }            &N_{k}    \\
\uTo^{i_{0*}} &\ruTo^{\phi_{f,k}}    &              \\
\pi_1(T_{f})   &                           &
\end{diagram}
\]
with $i_0$ and $i_1$ the appropriate inclusions.  This is not difficult to see.  Since $W$ and $\hat W$ have the same fundamental group we work with $W$.  $X$ deformation retracts to $T_{f}\times\{1\}$ and likewise $S^3\times[0,1]$ deformation retracts to $S^3\times\{0\}$.  
\begin{figure}[h!tp]
\centering
\begin{picture}(435,200)(0,0)
{\includegraphics[width=.98\textwidth]{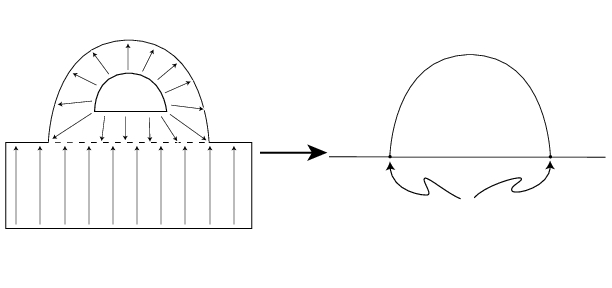}}
\put(-124,52){\small$\partial D\times S^1$}
\end{picture}
\caption{A deformation retraction of $W$.}
\end{figure}
Then removing a small ball from the interior of the second solid torus $S^1\times B^2$ that makes up $S^3$ (which does not affect fundamental group) we can further deformation retract to $T_{f}\times\{1\}$ union a disk, where the disk is glued in along $\gamma=[\{x_0\}\times S^1]$, yielding the desired result. 

\end{proof}

\begin{lemma}\label{lemma3}
The diagram
\[
\begin{diagram}
H_3(N_{1,k})              & \rTo^{i_*}  &H_3(N_{k}) \\
\dTo_{d^2_1}              & \qquad\qquad& \dTo_{d^2}    \\
H_1(\Sigma_1) \otimes \Lie_{1,k}& \rTo^{i_*}  &H_1(\Sigma) \otimes \Lie_{k}
\end{diagram}
\]
commutes
\end{lemma}

\begin{proof}
More generally we show that inclusion, $i:\Sigma_1\hookrightarrow\Sigma$, induces a morphism of spectral sequences.  As in \cite{abelian} we let
\begin{align*}
C_{1,*}&=\bigoplus_{n\ge0}C_n(N_{1,k+1})\quad\textrm{and} \\
C_*&=\bigoplus_{n\ge0}C_n(N_{k+1})
\end{align*}
be the normalized chain complexes for the groups $N_{1,k+1}$ and $N_{k+1}$ respectively.  That is, $C_{1,n}$ is the free abelian group generated by $n$-tuples of elements of $N_{1,k+1}$ modulo the relation that tuples containing the identity are zero.  The group ring $\z N_{1,k+1}$ acts component-wise: $g(\gamma_1,\ldots,\gamma_n)=(g\gamma_1,\ldots,g\gamma_n)$, and similarly for $C_*$.  The differentials $d^1:C_{1,n}\to C_{1,n-1}$ and $d:C_n\to C_{n-1}$ are described below.  We filter $C_{1,*}$ by  letting $C_{1,n}^j$ be the submodule of $C_{1,n}(N_{1,k+1})$ generated by $n$-tuples $(\gamma_1,\ldots,\gamma_n)$ where at least $(n-j)$ of the elements $\gamma_i$ belong to the subgroup $\Lie_{1,k}\subset N_{1,k+1}$.  We filter $C_*$ in exactly the analogous way.  As in \cite{abelian} p.\ 704, and \cite{HS} we may assume that the spectral sequences $(E^r_{p,q}(N_{1,k+1}), d^r_1)$ and $(E^r_{p,q}(N_{k+1}), d^r)$ are the ones associated to these explicit filtrations.  Note that $C_{1,*}$ and $C_*$ are modules over different rings, namely $\z N_{1,k+1}$ and $\z N_{k+1}$ respectively.  On the other hand we are only interested in the objects $E^r_{p,q}(N_{1,k+1})$ and $E^r_{p,q}(N_{k+1})$ as \emph{abelian groups}.  Therefore we view $C_{1,*}$ and $C_*$ as $\z$-modules.  Now we only need to check that $i_*:C_{1,*}\to C_*$ is a morphism of filtered graded differential $\z$-modules and it follows that $i_*:E^r_{p,q}(N_{1,k+1})\to E^r_{p,q}(N_{k+1})$ is a morphism of spectral sequences and in particular that $i_*$ commutes with the differentials $d^2_1$ and $d^2$ as in the above diagram (see \cite{McC} p.\ 48).
First $i:\Sigma_1\hookrightarrow\Sigma$ induces a map of group extensions:
\begin{diagram}
0 & \rTo & \Lie_{1,k} & \rTo & N_{1,k+1}  & \rTo & N_{1,k}    & \rTo & 1 \\
  &      & \dTo_{i_*} &      & \dTo_{i_*} &      & \dTo_{i_*} &      &   \\
0 & \rTo & \Lie_{k}   & \rTo & N_{k+1}    & \rTo & N_{k}      & \rTo & 1 
\end{diagram}
The induced map $i_*:C_{1,*}\to C_{*}$ given by $i_*(\gamma_1,\ldots,\gamma_n)=(i_*(\gamma_1),\ldots,i_*(\gamma_n))$ is obviously a morphism of graded $\z$-modules.  We must check that $i_*$ commutes with the differentials $d^1$ and $d$ in the sense that the following diagram commutes:
\[
\begin{diagram}
C_{1,n}          & \rTo^{i_*}   &C_{n} \\
\dTo_{d^1}       & \qquad\qquad & \dTo_{d}    \\
C_{1,n-1}        & \rTo^{i_*}   &C_{n-1}
\end{diagram}
\]
As in \cite{brown} p.\ 19 and \cite{HS} p.\ 118 the boundary operator $d^1$ is defined by:
\begin{align*}
d^1(\gamma_1,\ldots,\gamma_n)=&\gamma_1(\gamma_2,\ldots,\gamma_n)\\
                            &+\sum_{i=1}^{n-1}(-1)^i(\gamma_1,\ldots,\gamma_i\gamma_{i+1},\ldots,\gamma_n)\\
                            &+(-1)^{n}(\gamma_1,\ldots,\gamma_{n-1})
\end{align*}
and similarly for $d$.  It is routine to check that $i_*\circ d^1=d\circ i_*$.  It only remains to check that $i_*$ respects the filtrations.  We need $i_*(C_{1,n}^j)\subset C_{n}^j$.  This is obvious from the definition of the filtrations since $i_*(\Lie_{1,k})\subset\Lie_k$.  Thus, inclusion does indeed induce a morphism of the appropriate filtered graded $\z$-modules and hence also induces a morphism of the spectral sequences as desired.
\end{proof}

In the above proof we have actually shown:

\begin{lemma}\label{morphism}
A morphism of central extensions
\[
\begin{diagram}
1&\rTo&H        &\rTo&G        &\rTo&K        &\rTo&1\\
 &    &\dTo^\phi&    &\dTo^\phi&    &\dTo^\phi&    & \\
1&\rTo&H'       &\rTo&G'       &\rTo&K'       &\rTo&1      
\end{diagram}
\]
induces a morphism of Hochschild-Serre spectral sequences.  In particular, from the second page of the spectral sequences we have the commutative diagram:
\[
\begin{diagram}
H_p(K)\otimes H_q(H)        &\rTo^{\phi_*}&H_p(K')\otimes H_q(H') \\
\dTo^{d^2}                  &             & \dTo^{{d'}^2}         \\
H_{p-2}(K)\otimes H_{q+1}(H)&\rTo^{\phi_*}&H_{p-2}(K')\otimes H_{q+1}(H')
\end{diagram}
\]
\end{lemma}

The above lemmas combine to give the following:

\begin{proof}[Proof of proposition~\ref{diagram}]
We have the diagram
$$
\begin{diagram}
    &\qquad\qquad\qquad\qquad\qquad&                                &         &H_3(N_{k})   \\
    &                                                &                                &\ruTo^{i_*}\ruTo(4,3)^{\bar\tau_k }&   \\
    &                                                &H_3(N_{1,k})              &       &                 \\
\J(k)&\ruTo(2,1)_{\bar\tau_{1,k} }            &\dTo_{d_1^2}             &      &\dTo_{d^2}  \\
   &\rdTo(2,1)^{\tau_{1,k}}\rdTo(4,3)_{\tau_k}&H_1(\Sigma_1) \otimes \Lie_{1,k}&     &                \\
    &                                                &                                & \rdTo^{i_*}&                 \\
    &                                                &                                &     &H_1(\Sigma) \otimes \Lie_{k}
\end{diagram}
$$
Thinking of this diagram as a graph in the plane (groups as vertices and homomorphisms as edges), the graph cuts the plane into 4 bounded regions (3 triangles and one quadrilateral) and one unbounded region.  Lemmas~\ref{heap}, \ref{lemma1}, \ref{lemma2} and \ref{lemma3} have shown that each of the bounded regions commutes.  It follows that the unbounded region also commutes.
\end{proof}
\end{chapter}

\begin{chapter}{ Handlebodies and $\tau_k$}\label{handlebodiesandtau}

\section{The behavior of $\tau_k$ on homeomorphisms that extend to handlebodies}\label{sec:behavior}
Let $j:\Sigma\hookrightarrow\h$ be a handlebody bounded by $\Sigma$.  We use the notation $\bar\Lie_{k}:=\Lie_k(\pi_1(\h))=\pi_1(\h)_k/\pi_1(\h)_{k+1}$ and $\bar N_k:=\pi_1(\h)/\pi_1(\h)_{k}$.  There is an inclusion induced map $j_*:H_1(\Sigma) \otimes \Lie_{k}\to H_1(\h) \otimes \bar\Lie_{k}$.  We have: 

\begin{proposition}\label{ker}
If the homeomorphism $[f]\in\J(k)$ extends to the handlebody $\h$, then $j_*\circ\tau_k[f]=0\in H_1(\h) \otimes \bar\Lie_{k}$.
\end{proposition}

It immediately follows that:

\begin{corollary}\label{robust}
Let $[f]\in\J(k)$.
If for every handlebody $j:\Sigma\hookrightarrow\h$ bounded by $\Sigma$ we have $j_*\circ\tau_k[f]\ne0$, then $f$ does not extend to any handlebody.
\end{corollary}

In order to prove proposition~\ref{ker} we need the following lemma.

\begin{lemma}\label{Ftorus}
If $F:\h\to\h$ is a homeomorphism that restricts to $f:\Sigma\to\Sigma$ on the boundary with $[f]\in\J(k)$, then
$F_*$ acts as the identity on $\pi_1(\h)/\pi_1(\h)_k$.
\end{lemma}
\begin{proof}
There is a commutative diagram,
\[
\begin{diagram}
\pi_1(\Sigma)      & \rTo^{f_*}  & \pi_1(\Sigma)            \\
\dTo_{j_*} &                  & \dTo_{j_*}             \\
\pi_1(\h)     & \rTo^{F_*}   &\pi_1(\h)
\end{diagram}
\]
in which $j_*$ is surjective.  As in Proposition~\ref{prop:altfilt}, we need to check that $F_*(x)x^{-1}\in\pi_1(\h)_k$ for all $x\in\pi_1(\h)$.  Let $\bar x\in\pi_1(\Sigma)$ such that $j_*(\bar x)=x$.  Since $[f]\in\J(k)\subset\J'(k)$ we have $f_*(\bar x)\bar x^{-1}\in\pi_1(\Sigma)_k$.  Thus, $F_*(x)x^{-1}=F_*(j_*(\bar x))j_*(\bar x)^{-1}=j_*\circ f_*(\bar x)j_*(\bar x^{-1})=j_*(f_*(\bar x)\bar x^{-1})\in j_*(\pi_1(\Sigma)_k)\subset\pi_1(\h)_k$.
\end{proof}
We have the mapping torus: $T_F:=\h\times[0,1]/(x,0)\sim(F(x),1)$.  It follows from Lemma~\ref{Ftorus} that
\[
\frac{\pi_1(T_F)}{\pi_1(T_F)_k}\cong \bar N_{k}\times\z.
\]
And there is a continuous map $\Phi_{F,k}:T_F\to K(\bar N_{k},1)$ inducing the canonical epimorphism $\pi_1(T_F)\to\bar N_{k}$.
Next we note that there is a Hochschild-Serre spectral sequence associated to the central extension
\[
0\to\bar\Lie_{k}\to\bar N_{k+1}\to\bar N_{k}\to1
\]
which gives a differential $\bar d^2:H_3(\bar N_{k} )\to H_1(\h) \otimes \bar\Lie_{k}$.  We have that
the diagram
\[
\begin{diagram}
H_3(N_{k} )   & \rTo^{j_*}  & H_3(\bar N_{k} )            \\
\dTo_{d^2}            &                  & \dTo_{\bar d^2}             \\
H_1(\Sigma) \otimes \Lie_{k} & \rTo^{j_*}   &H_1(\h) \otimes \bar\Lie_{k}
\end{diagram}
\]
commutes, by Lemma~\ref{morphism}.


\begin{proof}[Proof of Proposition~\ref{ker}]
Suppose that $f:\Sigma\to\Sigma$ extends to a homeomorphism $F:\h\to\h$.  $T_F$ is a 4-manifold with boundary $\partial T_F=T_f$.  Let $I:T_f\hookrightarrow T_F$ be inclusion.  Recall that $\phi_{f,k}:T_f\to K(N_{k},1)$ is a continuous map induced by the surjection $\pi_1(T_f)\to N_{k}$.  Then there is a continuous map $\iota$ induced by the inclusion induced homomorphism $j_*:N_k\to\bar N_k$ such that the diagram
\[
\begin{diagram}
T_f         & \rTo^{\phi_{f,k}}   & K(N_{k},1)           \\
\dTo_{I} &                         & \dTo_{\iota }             \\
T_F        & \rTo^{\Phi_{F,k} } &K(\bar N_{k},1)
\end{diagram}
\]
commutes up to homotopy of the maps, which yields the commutative diagram:
\[
\begin{diagram}
H_3(T_f)         & \rTo^{\phi_{f,k*}}   & H_3(N_{k})      \\
\dTo_{I_*} &                         & \dTo_{j_*}             \\
H_3(T_F)        & \rTo^{\Phi_{F,k*} } &H_3(\bar N_{k})
\end{diagram}
\]

Let $C\in C_3(T_f)$ be a 3-chain representing the fundamental class $[T_f]$.  That is $[C]=[T_f]\in H_3(T_f)$.  Then, $j_*\circ\bar\tau_k[f]=j_*\circ\phi_{f,k*}[C]=\Phi_{F,k*}\circ I_*[C]$, but $I_*[C]=0\in H_3(T_F)$, since $T_f=\partial T_F$ as already noted. 
\end{proof}

\section{Robust homeomorphisms}\label{robusthomeos}

\begin{definition}
If a homeomorphism $[f]\in\J(k)$ satisfies the hypothesis of Corollary~\ref{robust} we call $f$ a \emph{robust} homeomorphism (or $k$-robust for clarity).  That is, $f$ is robust if for all handlebodies $j:\Sigma\hookrightarrow\h$ bounded by $\Sigma$, $j_*\circ\tau_k[f]\ne0\in H_1(\h) \otimes \bar\Lie_{k}$.
\end{definition}

By corollary~\ref{robust} it remains only to find robust homeomorphisms in order to prove Theorem~\ref{main}.  This section is devoted to such a proof.  First note that as the handlebody $\h$ varies, we are not so interested in the homomorphism $j_*\circ\tau_k:\J(k)\to H_1(\h) \otimes \bar\Lie_{k}$ as in its kernel.  In fact, $\tau_k$ is a fixed homomorphism and it is only $j_*$ that varies with $\h$, so we are really only interested in the kernel of $j_*:\im(\tau_k)\to H_1(\h) \otimes \bar\Lie_{k}$, where $\im(\tau_k)$ is the image of $\tau_k$.

\begin{lemma}\label{h1}
The kernel of the homomorphism $j_*:\im(\tau_k)\to H_1(\h) \otimes \bar\Lie_{k}$ depends only on the kernel of the homomorphism $j_*:H_1(\Sigma)\to H_1(\h)$.
\end{lemma}
\begin{proof}
We prove that the kernel of the homomorphism $j_*:\Lie_{k}\to \bar\Lie_{k}$ depends only on the kernel of $j_*:H_1(\Sigma)\to H_1(\h)$, from which the result follows.
The induced map $j_*:\pi_1(\Sigma_1)\to\pi_1(\h)$ is an epimorphism.  As in Section~\ref{sec:model}, $\pi_1(\Sigma_1)$ is a free group with free basis $\{a_1,\ldots,a_g,b_1,\ldots,b_g\}$.  We assume, without loss of generality, that $\{a_1,\ldots,a_g\}$ normally generate the kernel of the projection $\pi_1(\Sigma_1)\twoheadrightarrow\pi_1(\h)$.   Then $\{\alpha_1,\ldots,\alpha_g\}$ is a basis for the kernel of $j_*:H_1(\Sigma)\to H_1(\h)$.  
Theorem~\ref{thm:labute} gives us that the kernel of the induced map $\Lie_{(1)}\to\bar\Lie$ is exactly the Lie ring ideal $\bar{\mathfrak I}$ generated by $\{\alpha_1,\ldots,\alpha_g\}$.  See \cite{Labute} for a check of the hypotheses of Theorem~\ref{thm:labute} in this case.
Thus, we have that $\ker(j_*:\Lie_{(1)}\to\bar\Lie)$ depends only on $\ker(j_*:H_1(\Sigma)\to H_1(\h))$.  It only remains to observe that since $\Lie=\Lie_{(1)}/{\mathfrak I}$, as in Lemma~\ref{labute1}, and since ${\mathfrak I}\subset\bar{\mathfrak I}$, Proposition~\ref{prop:Liebasics} gives that $\ker(\Lie\twoheadrightarrow\bar\Lie)=\bar{\mathfrak I}/{\mathfrak I}$, which depends only on $\ker(j_*:H_1(\Sigma)\to H_1(\h))$.
\end{proof}

Lemma~\ref{h1} leads us to consider the set of all possible kernels $K\subset H_1(\Sigma)$ of inclusion induced maps into handlebodies.  For any handlebody $j:\Sigma\hookrightarrow\h$ bounded by $\Sigma$, the kernel of the map $j_*:H_1(\Sigma,\R)\to H_1(\h,\R)$ is a Lagrangian subspace of $H_1(\Sigma,\R)$.  The space of Lagrangian subspaces of $H_1(\Sigma,\R)$ (known as the Lagrangian Grassmannian) is a $g(g+1)/2$-dimensional smooth manifold (see e.g. \cite{eli}, p.\ 20), which we will denote by $\mathbb  L$.  Note that $\mathbb  L$ is larger than the set of all possible kernels $K\subset H_1(\Sigma)$ of inclusion induced maps.  For instance in genus one, $\mathbb  L$ is the space of lines through the origin in $\R^2$.  Only those lines with rational slope appear as the kernel of some $j_*:H_1(\Sigma)\to H_1(\h)$.  We will construct a vector bundle $E$ over $\mathbb  L$, and use a dimensional argument to prove: 
\begin{theorem}\label{robusthomeo}
Given $k\ge2$, there are $k$-robust homeomorphisms provided that the genus of $\Sigma$ is sufficiently large.
\end{theorem}

The remainder of this section is devoted to a proof of Theorem~\ref{robusthomeo}.
In Subsection~\ref{subsec:E} we construct the needed vector bundle.  In Subsection~\ref{subsec:reduce} we reduce the problem to a dimensional argument.  In Subsection~\ref{subsec:levine} we give an estimate for the appropriate dimension.  Finally, in Subsection~\ref{subsec:proof} we prove the theorem.

\subsection{Construction of the vector bundle $E$}\label{subsec:E}

Consider the trivial vector bundle $q:{\mathbb  L}\times\im(\tau_k)^\R\to{\mathbb  L}$, where $\im(\tau_k)^\R:=\im(\tau_k)\otimes\R$ (in general we will use the notation $A^\R:=A\otimes\R$).  Of course there is the projection map $p_2:{\mathbb  L}\times\im(\tau_k)^\R\to\im(\tau_k)^\R$.  We construct a subbundle $E\subset{\mathbb  L}\times\im(\tau_k)^\R$ by restricting the fibers.  In particular, above any point $K\in{\mathbb  L}$ we wish to define  the fiber ${\mathbb K }_K$ of $E$ and prove that $E$ is a smooth vector bundle.  
See Figure~\ref{fig:E} for a schematic of the construction of $E$.  
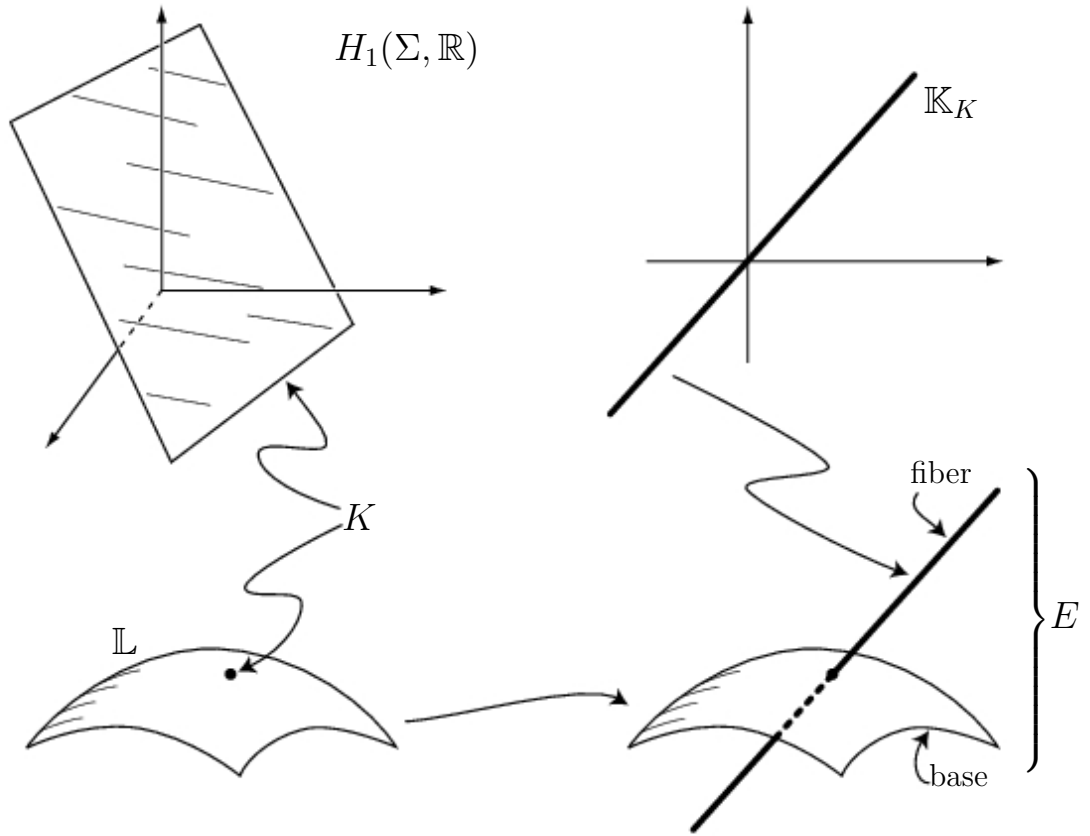
\begin{figure}[h!tp]
\centering
\begin{picture}(435,320)(0,0)
\includegraphics[width=.98\textwidth]{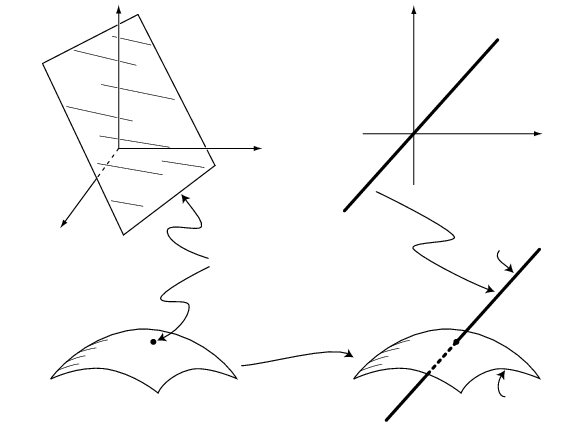}
\put(-362,75){\large$\mathbb L$}
\put(-275,122){\large$K$}
\put(-35,85){$\left. \phantom{\begin{array}{c}0\\0\\0\\0\\0\\0\\0\\0\end{array}} \right\}${\large$E$}}
\put(-278,300){\large$H_1(\Sigma,\R)$}
\put(-55,280){{\large${\mathbb K }_K$}}
\put(-53,25){base}
\put(-60,140){fiber}
\end{picture}
\caption{Schematic of the construction of the vector bundle $E$.}\label{fig:E}
\end{figure}
Suppose that ${\mathcal B}=\{x_1,\ldots x_g,y_1,\ldots,y_g\}$ is a symplectic basis for $H_1(\Sigma)$.  
We define a grading on $\Lie_k$ associated to ${\mathcal B}$.
According to Section~\ref{sec:freeLierings}, $\Lie_{1,k}$ has a basis consisting of basic commutators of weight $k$ in the elements of ${\mathcal B}$.  Let ${\mathcal B}_k$ be the set of said basic commutators.  We partition ${\mathcal B}_k$:
\[
{\mathcal B}_k=\coprod_{i=0}^k({\mathcal B}_k)_i,
\]
where $({\mathcal B}_k)_i$ consists of those basic commutators with exactly $i$ appearances of the $x$'s.  This partition of the basis ${\mathcal B}_k$ gives a grading on the free abelian group $\Lie_{1,k}$.  Namely, we define: 
\[
(\Lie_{1,k})_i:={\mathrm{span}}_\z ({\mathcal B}_k)_i.
\]
That is,
\[
\Lie_{1,k}=\bigoplus_{i=0}^k(\Lie_{1,k})_i.
\]
This gives a bigrading of $\Lie_{(1)}$.
The symplectic class, $\Omega=\sum_{i=1}^g [\alpha_i,\beta_i]=\sum_{i=1}^g[x_i,y_i]$, is the defining relator for $\Lie$.  Since $\Omega$ lies in one of the bigraded pieces (namely $(\Lie_{1,2}^\R)_1$), this bigrading of $\Lie_{(1)}$ descends to a bigrading of $\Lie$ and a grading:
\[
\Lie_k=\bigoplus_{i=1}^k(\Lie_k)_i.
\]
We have a basis for the free abelian group $H_1(\Sigma)\otimes\Lie_{1,k}$, namely ${\mathcal B}\otimes {\mathcal B}_k$, which may also be partitioned into sets whose elements contain exactly $i$ occurrences of the $x$'s.  So similarly we have a grading on $H_1(\Sigma)\otimes\Lie_{1,k}$, which descends to a grading:
\[
H_1(\Sigma)\otimes\Lie_{k}=\bigoplus_{i=0}^{k+1}\big(H_1(\Sigma)\otimes\Lie_{k}\big)_i.
\]
Next we define a grading on $\im(\tau_k)$:
\[
\im(\tau_k)_i:=\im(\tau_k)\cap\big(H_1(\Sigma)\otimes\Lie_{k}\big)_i,
\]
and this gives a grading on the vector space:
\[
\im(\tau_k)^\R_i=\im(\tau_k)_i\otimes\R.
\]
\begin{proposition}\label{prop:taugrading}
For $k\ge 2$, the above definition of $\im(\tau_k)^\R_i$ does indeed give a grading.  That is
\[
\im(\tau_k)^\R=\bigoplus_{i=0}^{k+1}\im(\tau_k)^\R_i.
\]
\end{proposition}
In order to prove Proposition~\ref{prop:taugrading} we need the following theorem of Morita \cite{derivations}:
\begin{theorem}[Morita]\label{thm:tauisalg}
The free abelian group
\[
\im(\tau):=\bigoplus_{k=2}^\infty\im(\tau_k)
\]
has a natural Lie ring structure.
\end{theorem}
We do not prove Theorem~\ref{thm:tauisalg} here, but we would like to give an idea of how the Lie ring structure on $\im(\tau)$ arises.
In fact, $\im(\tau)$ is a subring of the Lie ring of positive degree derivations\footnote{A derivation of a Lie ring $L$, is a linear map $D:L\to L$ such that $D([x,y])=[D(x),y]+[x,D(y)]$ for all $x,y\in L$.} 
of the graded Lie ring $\Lie$.  Recall that $\im(\tau_k)\subset H_1(\Sigma)\otimes\Lie_k$, which we canonically identify with $\Hom(H_1(\Sigma),\Lie_k)$.  Thus let $\phi\in\Hom(H_1(\Sigma),\Lie_k)$.  We define a derivation of $\Lie$ associated to $\phi$.  Following Morita we denote this derivation $\phi\{\cdot\}$ (whereas the homomorphism $\phi$ will be denoted by $\phi(\cdot)$).  For $\ell=[\cdots[\ell_1,\ell_2],\cdots,],\ell_k]\in\Lie_k$ we define
\[
\phi\{\ell\}:=\sum_{i=1}^k[\cdots[\ell_1,\ell_2],\cdots],\phi(\ell_i)],\cdots],\ell_k].
\]
This shows how elements of $\im(\tau)$ are considered derivations of $\Lie$.  The bracket of two such derivations $[\phi,\psi]$ is defined by:
\[
[\phi,\psi]\{\ell\}:=\phi\{\psi\{\ell\}\}-\psi\{\phi\{\ell\}\}.
\]

Thus, by Remark~\ref{lem:liealg}
\[
\im(\tau)^\R=\bigoplus_{k=2}^\infty\im(\tau_k)^\R
\]
is a Lie algebra.
We need the following theorem of Hain (see \cite{infinitesimal}) whose proof we omit:
\begin{theorem}[Hain]\label{thm:taugen}
The Lie algebra
\[
\im(\tau)^\R=\bigoplus_{k=2}^\infty\im(\tau_k)^\R
\]
is generated by the degree two summand $\im(\tau_2)^\R$.
\end{theorem}
\begin{proof}[Proof of Proposition~\ref{prop:taugrading}]
Johnson described $\im(\tau_{1,2})$ in \cite{J80} as 
\[
\im(\tau_{1,2})=\bigwedge^3H_1(\Sigma)\subset H_1(\Sigma)\otimes \left(\bigwedge^2H_1(\Sigma)\right). 
\]
Note that $\Lie_{1,2}\cong\bigwedge^2H_1(\Sigma)$.
It is clear that $\im(\tau_{1,2})=\bigwedge^3H_1(\Sigma)$ admits the described grading.  That is
\[
\im(\tau_{1,2})=\bigoplus_{i=0}^3 \im(\tau_{1,2})_i,
\]
where
\[
\im(\tau_{1,2})_i=\big(H_1(\Sigma)\otimes\Lie_{1,2}\big)_i\cap\im(\tau_{1,2}).
\]
The degree two summand of $\im(\tau)^\R$, is given by
\[
\im(\tau_2)^\R=\frac{\bigwedge^3H_1(\Sigma,\R)}{\Omega\wedge H_1(\Sigma,\R)}.
\]
The grading on $\im(\tau_{1,2})$ gives a grading on $\im(\tau_{1,2})^\R$ which descends to a grading on $\im(\tau_2)^\R$.
One may easily check from the above definition of the bracket operation that $[\im(\tau_k)^\R_i,\im(\tau_l)^\R_j]\subset \im(\tau_{k+l-1})^\R_{i+j-1}$, where $\im(\tau_{k+l-1})^\R_{-1}=\im(\tau_{k+l-1})^\R_{(k+l+1)}=0$.  Thus, starting from a basis for $\im(\tau_2)^\R$ adapted to its grading, we get a basis for $\im(\tau_k)^\R$ adapted to its grading by Theorem~\ref{thm:taugen}.
\end{proof}

We are now ready to define ${\mathbb K }_K$.
$K$ is a Lagrangian subspace of $H_1(\Sigma,\R)$. We let ${\mathcal B}=\{x_1,\ldots x_g,y_1,\ldots,y_g\}$ be a symplectic basis adapted to $K$.  That is,  $\{x_1,\ldots x_g\}$ is a basis for $K$.  Grade $\im(\tau_k)^\R$ as above according to ${\mathcal B}$.  We define:
\[
{\mathbb K }_K:=\bigoplus_{i=1}^{k+1}\im(\tau_k)^\R_i
\]
\begin{proposition}
$E$ is a smooth vector bundle. 
\end{proposition}
\begin{proof}
Let $\mathcal S$ be the universal bundle over $\mathbb L$.  That is, $\mathcal S$ is the subbundle of ${\mathbb L}\times H_1(\Sigma,\R)$ whose fiber over $K\in{\mathbb L}$ is $K$.  It is well known that $\mathcal S$ is a smooth (in fact holomorphic) vector bundle.  Let $\mathcal Q$ be the universal quotient bundle, that is the bundle over ${\mathbb L}$ whose fiber over $K$ is $H_1(\Sigma,\R)/K$.  $\mathcal Q$ is also a smooth vector bundle.  We have a short exact sequence of vector bundles:
\[
\begin{diagram}
0 & \rTo & {\mathcal S} & \rTo & {\mathbb L}\times H_1(\Sigma,\R) & \rTo & {\mathcal Q} & \rTo & 0 \\
  &      &              & \rdTo&  \dTo^p                            & \ldTo&              &      &   \\
  &      &              &      & {\mathbb L}                      &      &              &      &   \\
\end{diagram}
\]
This sequence is self-dual in the sense that there is a canonical commutative diagram:
\[
\begin{diagram}
0 & \rTo & {\mathcal S} & \rTo & {\mathbb L}\times H_1(\Sigma,\R) & \rTo & {\mathcal Q} & \rTo & 0 \\
  &      & \dTo^{\cong} &      &   \dTo^{\cong}                   &      & \dTo^{\cong} &      &   \\
0 & \rTo & \check{\mathcal Q} & \rTo & {\mathbb L}\times \check H_1(\Sigma,\R) & \rTo & \check{\mathcal S} & \rTo & 0 \\
\end{diagram}
\]
where $\check V$ denotes the dual of the vector space $V$.
Thus, we have a canonical isomorphism ${\mathbb L}\times H_1(\Sigma,\R)\cong{\mathcal S}\oplus{\mathcal Q}$.
Let $\{U_i\}_{i\in I}$ be an open cover of $\mathbb L$ which trivializes the universal subbundle and universal quotient bundle.  For $i\in I$, let ${\mathcal B}(u)=\{x_1(u),\ldots,x_g(u),y_1(u),\ldots,y_g(u)\}_{u\in U_i}$ be a smoothly varying basis for $H_1(\Sigma,\R)$ adapted to the decomposition $p^{-1}(U_i)\cong (U_i\times K)\oplus(U_i \times \check K)$.  That is $\{x_1(u),\ldots,x_g(u)\}$ is a basis for $K$.  One obtains from ${\mathcal B}(u)$ a smoothly varying basis for $\im(\tau_{k})^\R$ adapted to its grading, and thus, a smooth trivialization of $(q|_E)^{-1}(U_i)$.
\end{proof}

\begin{lemma}
If $K\in{\mathbb  L}$ is the kernel of an inclusion induced map $j_*:H_1(\Sigma,\R)\to H_1(\h,\R)$, that is, if $K=L^\R$, where $L$ is a Lagrangian subspace of $H_1(\Sigma,\z)$, then ${\mathbb K }_K=\ker(j_*:\im(\tau_k)^\R\to H_1(\h,\R)\otimes\bar\Lie_k)$.
\end{lemma}
\begin{proof}
As in the proof of Lemma~\ref{h1}, $\ker(j_*:\Lie_k\to \bar\Lie_k)$ is the $k^{\textrm{th}}$ graded piece of the Lie ring ideal generated by $L$,   which---it is clear from the definition---is exactly:
\[
(\Lie_k)_K:=\bigoplus_{i=1}^k(\Lie_k)_i.
\]
Thus, $\ker(j_*:H_1(\Sigma)\otimes\Lie_k\to H_1(\h)\otimes\bar\Lie_k)$ is
\[
(L\otimes\Lie_k)\oplus\big( H_1(\Sigma)\otimes (\Lie_k)_K\big)=\bigoplus_{i=1}^{k+1}\big(H_1(\Sigma)\otimes\Lie_{k}\big)_i
\]
and
\begin{align*}
\ker\big(j_*:\im(\tau_k)\to H_1(\h)\otimes\bar\Lie_k\big)=&\im(\tau_k)\cap\bigoplus_{i=1}^{k+1}\big(H_1(\Sigma)\otimes\Lie_{k}\big)_i\\
=&\bigoplus_{i=1}^{k+1}\im(\tau_k)_i.
\end{align*}
Therefore,
\begin{align*}
\ker\big(j_*:\im(\tau_k)^\R\to H_1(\h,\R)\otimes\bar\Lie_k\big)=&\bigoplus_{i=1}^{k+1}\im(\tau_k)_i^\R\\
=&{\mathbb K}_K.
\end{align*}
\end{proof}

\subsection{Reduction to a dimensional argument}\label{subsec:reduce}

By restriction, there is a smooth map $P=p_2\vert_E:E\to \im(\tau_k)^\R$.
Notice that if a vector $v\in\im(\tau_k)^\R$ is in the kernel of some inclusion induced homomorphism $j_*:\im(\tau_k)^\R\to H_1(\h,\R) \otimes \bar\Lie_{k}^\R$, then $v$ is in the image of $P$.  Therefore, to prove the existence of robust homeomorphisms it suffices to show that $P$ does not hit all of the points of $\im(\tau_k)$, which we think of as an integer lattice in $\im(\tau_k)^\R$.   
Restricted to a fiber, $q|_{E}^{-1}(K)$, $P$ is just the inclusion ${\mathbb K}_K\hookrightarrow \im(\tau_k)$.  In other words, we can write $\im(P)$ as
\[
\im(P)=\bigcup_{K\in {\mathbb L}}{\mathbb K}_K.
\]
Thus, for any $v\in\im(P)$, there is some $K\in{\mathbb L}$ such that $v\in{\mathbb K}_K$, which is a subspace of $\im(\tau_k)^\R$.  Hence, $rv\in{\mathbb K}_K\subset\im(P)$ for any real number $r$.
That is, a non-zero vector $v\in\im(\tau_k)^\R$ is in the image of $P$ if and only if $rv$ is in the image of $P$, for all non-zero, real numbers $r$.  It therefore suffices to show that there is a rational point $v\in\im(\tau_k)^{\mathbb  Q}\subset\im(\tau_k)^\R$ not hit by $P$, because then $Nv\in\im(\tau_k)\subset\im(\tau_k)^{\mathbb  Q}\subset\im(\tau_k)^\R$ is an integer point (for some positive integer $N$) not hit by $P$.  
\begin{proposition}
To prove the existence of robust homeomorphisms, it suffices to show:
\begin{equation}\label{ineq}
\dim E<\dim\im(\tau_k)^\R .
\end{equation}
\end{proposition}
\begin{proof}
Since $P$ is linear on the fibers we may projectivize the fibers of $E$ (call this projectivized space $\bar E$) as well as the target space $H_1(\h,\R)\otimes\Lie_k$ (call this $\bar V$) and $P$ descends to a map $\bar P:\bar E\to \bar V$.  Since $\bar E$ is compact so is the image of $\bar P$.  Thus $\im(\bar P)$ is closed, but by the Morse-Sard theorem (see, for instance, \cite{hirsch}) and since $\bar P$ is smooth, $\bar V-\im(\bar P)$ is non-empty (in fact dense), and thus, contains a rational point.  By the above discussion, this is sufficient to prove the proposition.
\end{proof}

We will eventually prove Inequality~\eqref{ineq} for large genera, but first we manipulate it into a more workable form. The dimension of the vector bundle $E$ is the sum of the dimensions of the base and fiber:
\[
\dim E=\dim{\mathbb  L}+\dim\ker\big(j_*:\im(\tau_k)^\R\to H_1(\h,\R) \otimes \bar\Lie_{k}\big),
\]
where we have restricted attention to a fixed (but arbitrary) fiber.  That is to say, we are considering a fixed handlebody $j:\Sigma\hookrightarrow\h$ bounded by $\Sigma$.  In fact we will see in the following subsection that we wish to choose $j:\Sigma\hookrightarrow\h$ such that the curves $a_i$ bound disks in $\h$.  Throughout the remainder of this chapter $\h$ will refer to this fixed handlebody.  We can rewrite the dimension of the above kernel using the rank-nullity theorem:
\[
\dim\ker\big(j_*:\im(\tau_k)^\R\to H_1(\h,\R) \otimes \bar\Lie_{k}^\R\big)=\dim\im(\tau_k)^\R-\dim\im(j_*\circ\tau_k)^\R.
\]
Then Inequality~\eqref{ineq} becomes
\[
\dim{\mathbb L}+\dim\im(\tau_k)^\R-\dim\im(j_*\circ\tau_k)^\R<\dim\im(\tau_k)^\R.
\]
Eliminating $\dim\im(\tau_k)^\R$  from each side and using $\dim{\mathbb L}=g(g+1)/2$ we get the equivalent inequality
\begin{equation}\label{simpineq}
\frac{g(g+1)}{2}<\dim\im(j_*\circ\tau_k)^\R.
\end{equation}
Thus, we need an estimate on $\dim\im(j_*\circ\tau_k)^\R$.

\subsection{Levine's estimate for $\dim\im(j_*\circ\tau_k)^\R$}\label{subsec:levine}

We use an estimate provided by Levine in \cite{L} to prove Inequality~\eqref{simpineq}.  Denote by $S_{g+1}$ the sphere with $g+1$ boundary components.  The framed pure braid group on $g$ strands, which we denote by $\p^{\mathrm{fr}}_g$, can be defined as the mapping class group of $S_{g+1}$ \emph{relative} to the boundary (i.e. homeomorphisms and isotopies are required to fix the boundary pointwise).  If we denote the regular pure braid group by $\p_g$, then there is a canonical isomorphism $\p^{\mathrm{fr}}_g\cong\p_g\oplus\z^g$.  Thus, we may think of $\p_g$ as a subgroup of $\p^{\mathrm{fr}}_g$.
For any embedding of $S_{g+1}$ into the surface $\Sigma_1$, there is an induced homomorphism from $\p^{\mathrm{fr}}_g$ to the mapping class group $\Mod(\Sigma,D)$.  Thus, we also have an induce homomorphism $\p_g\to\Mod(\Sigma,D)$.  
Let $\h'$ be the genus $g$ handlebody $S_{g+1}\times[0,1]$.  We consider $S_{g+1}$ to be identified with $S_{g+1}\times\{1\}$ in $\partial\h'$. Pick a homeomorphism from $\h'$ to $\h$ such that $S_{g+1}$ misses the interior of the disk $D$ in $\Sigma=\partial\h$.  This gives an embedding of $S_{g+1}$ into $\Sigma_1$.  We consider $\h'$ and $\h$ to be identified via this homeomorphism.  We denote by $\theta:\p_g\to\Mod(\Sigma,D)$ the induced homomorphism.
  
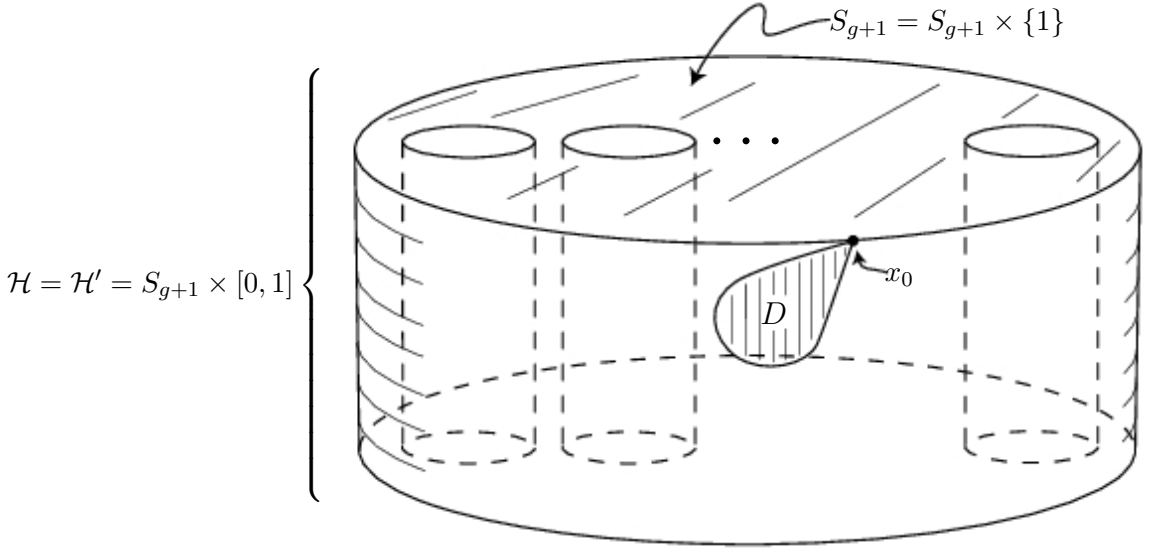
\begin{figure}[h!tp]
\centering
\begin{picture}(435,210)(0,0)
\includegraphics[width=.98\textwidth]{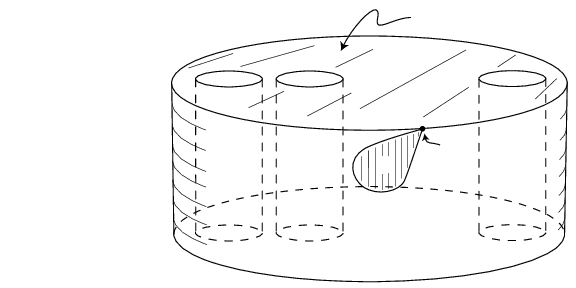}
\put(-435,100){\small$\h=\h'=S_{g+1}\times[0,1]\left\{\phantom{\begin{array}{c}0\\0\\0\\0\\0\\0\\0\\0\\0\\0\\0\\0\end{array}}\right.$}
\put(-150,88){$D$}
\put(-103,104){\small$x_0$}
\put(-124,199){\small$S_{g+1}=S_{g+1}\times\{1\}$}
\put(-170,151){\Huge$\cdots$}
\end{picture}
\caption{Embedding $S_{g+1}$ into $\Sigma=\partial\h$.}\label{disk}
\end{figure}

The following lemma is attributed to Oda.  We omit the proof, but it is given in \cite{L} and \cite{GH}.
\begin{lemma}[Oda]\label{oda}
The homomorphism $\theta:\p_g\to\Mod(\Sigma,D)$ is an embedding and furthermore, $\theta$ induces embeddings $\rho:(\p_g)_k/(\p_g)_{k+1}\to\J(k)/\J(k+1)$, for each $k$.  Equivalently, (since $\ker\tau_{1,k}=\J(k+1)$) $\tau_{1,k}\circ\rho:(\p_g)_k/(\p_g)_{k+1}\to H_1(\Sigma) \otimes \Lie_{k}$ is an embedding.
\end{lemma}
Recall that we have defined two inclusion maps, $i:\Sigma_1\hookrightarrow\Sigma$ and $j:\Sigma\hookrightarrow\h$.  Let $\tilde j:\Sigma_1\hookrightarrow\h$ be the composition $j\circ i$.
Levine has shown that:
\begin{lemma}[Levine]\label{levine}
\[
\tilde j_*\circ\tau_{1,k}\circ\rho:(\p_g)_k/(\p_g)_{k+1}\to H_1(\h) \otimes \bar\Lie_{k}
\]
is an embedding.  
\end{lemma}
\begin{proof}
Let $B_1,\ldots,B_{g+1}$ be the components of $\partial S_{g+1}$.  Select a basepoint $x_0'$ in $B_{g+1}$.  For convenience let us assume that $x_0'$ gets mapped to the basepoint $x_0$ and thus identify these two points.  Likewise let $x_i$ be a point on $B_i$ for $i=1,\ldots,g$.  Choose simple, pairwise disjoint arcs $A_i$ running  from $x_0$ to $x_i$.  
\begin{figure}[h!tp]
\centering
\begin{picture}(435,210)(0,0)
\includegraphics[width=.98\textwidth]{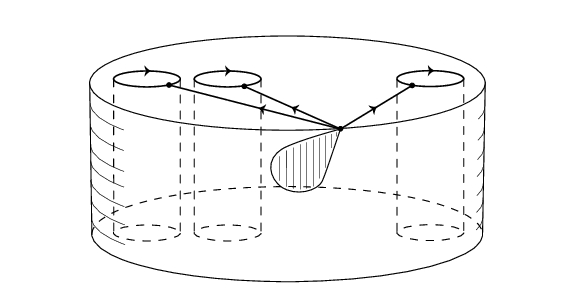}
\put(-176,111){\small$x_0$}
\put(-315,143){\small$x_1$}
\put(-256,143){\small$x_2$}
\put(-130,143){\small$x_g$}
\put(-253,125){$A_1$}
\put(-211,135){$A_2$}
\put(-154,128){$A_g$}
\put(-325,167){$B_1$}
\put(-260,167){$B_2$}
\put(-128,168){$B_g$}
\put(-235,151){\Huge$\cdots$}
\end{picture}
\caption{The arcs $A_i$ and closed curves $B_i$ used to define $a'_i$ and $b'_i$.}\label{fig:AB}
\end{figure}
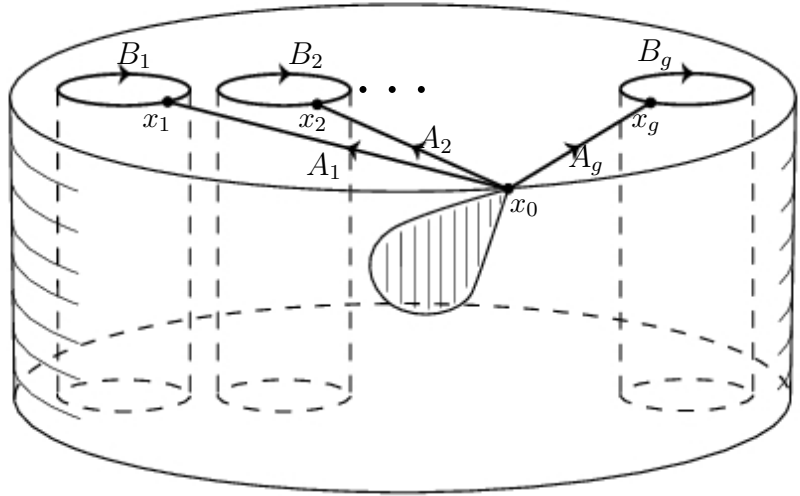
Let $a'_i$ be the oriented simple closed curve obtained by following $A_i (=A_i\times\{1\})$ to $x_i$, then following the arc $\{x_i\}\times[0,1]$, followed by $A_i\times\{0\}$ and finally $\{x_0\}\times[0,1]$ arriving at the basepoint $x_0 (=x_0\times1)$.  See Figure~\ref{fig:ab}.
Let $b'_i$ be the oriented simple closed curve based at $x_0$ obtained by traversing the arc $A_i$, going around $B_i$ (in a direction such that $a'_i\cdot b'_i=1$) and following $A_i$ back to $x_0$.    
\begin{figure}[h!tp]
\centering
\begin{picture}(435,210)(0,0)
\includegraphics[width=.98\textwidth]{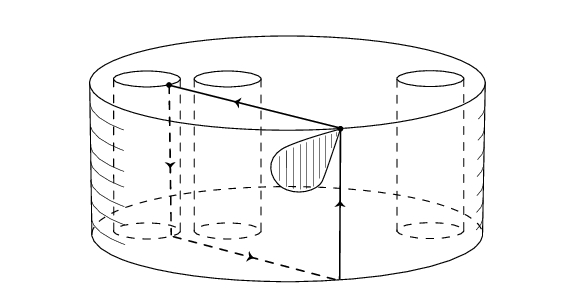}
\put(-175,55){$a'_1$}
\put(-235,151){\Huge$\cdots$}
\end{picture}
\caption{The curve $a'_1$.}\label{fig:ab}
\end{figure}

We assume, without loss of generality, that $a'_i$ is homotopic to $a_i$ and that $b'_i$ is homotopic to $b_i$.
Thus, $\pi_1(\Sigma_1)$ is freely generated by $\{a_1,\ldots,a_g,b_1,\ldots,b_g\}$, and  $\pi_1(\h)=\pi_1(S_{g+1})$ (identified by the inclusion-induced map) is freely generated by $\{b_1,\ldots,b_g\}$.  Of course the $a_i$ bound disks in $\h$ and normally generate the kernel of the inclusion-induced map $\pi_1(\Sigma_1)\to\pi_1(\h)$.  Since the union of arcs $\cup A_i$ cut $S_{g+1}$ into a disk we may use the Alexander method (see \cite{primer}) to characterize its (relative) mapping classes.  In particular, an element $x\in\p^{\mathrm{fr}}_g$ is determined by an automorphism of $\pi_1(S_{g+1})$ of the form:
\begin{equation}\label{eq:auto}
b_i\mapsto\lambda_ib_i\lambda_i^{-1},
\end{equation}
where the $\lambda_i$ are elements of $\pi_1(S_{g+1})$ satisfying: 
\[
\lambda_1b_1\lambda_1^{-1}\cdots\lambda_gb_g\lambda_g^{-1}=b_1\cdots b_g.
\]
See \cite{L}.  We take $x$ to be in the subgroup $\p_g$.  The mapping class $\theta(x)\in\Mod(\Sigma,D)$ induces the automorphism
\[
b_i\mapsto\lambda_ib_i\lambda_i^{-1},\quad a_i\mapsto \lambda_ia_i.
\]
Assume that $\theta(x)\in\J(k)$.  It follows that $a_i\mapsto \lambda_i a_i$ must be the identity modulo $\pi_1(\Sigma)_{k}$.  That is, $\lambda_i\in\pi_1(\Sigma)_{k}$.  We compute $\sigma\circ\theta(x)$, where $\sigma:\J(k)\to\Hom(H_1(\Sigma),\Lie_k)$ is as defined in Section~\ref{sec:johnsonhom}.  We have that $\sigma\circ\theta(x)$ is the homomorphism:
\begin{equation}\label{sigma}
\beta_i\mapsto[[\lambda_i, b_i]],\quad \alpha_i\mapsto [\lambda_ia_ia_i^{-1}]=[\lambda_i],
\end{equation}
where $\alpha_i$ and $\beta_i$ are the homology classes of $a_i$ and $b_i$, respectively. Note that we have used $[\ ,\ ]$ to mean the commutator and $[\ \ ]$ to mean equivalence class in $\Lie_k$.  Also note that $[\lambda_i, b_i]\in[\pi_1(\Sigma_1),\pi_1(\Sigma_1)_k]=\pi_1(\Sigma_1)_{k+1}$.  Thus, the homomorphism \eqref{sigma} simplifies to:
\[
\beta_i\mapsto0,\quad\alpha_i\mapsto[\lambda_i].
\]
We now compute $\tau_k\circ\theta(x)=\eta^{-1}\circ\sigma\circ\theta(x)$.  We claim that
\[
\tau_k\circ\theta(x)=-\sum_i\beta_i\otimes[\lambda_i].
\]
By the definition given in Section~\ref{sec:johnsonhom} we have
\begin{align*}
\eta\left(-\sum_i\beta_i\otimes[\lambda_i]\right)(\alpha_j)&=-\sum_i\omega(\beta_i,\alpha_j)[\lambda_i]=[\lambda_j],\quad\textrm{and}\\
\eta\left(-\sum_i\beta_i\otimes[\lambda_i]\right)(\beta_j)&=-\sum_i\omega(\beta_i,\beta_j)[\lambda_i]=0,
\end{align*}
verifying the claim.  Since the $\lambda_i$ are words in the $b_i$, it follows that if $\sum_i\beta_i\otimes[\lambda_i] $ is not zero, then neither is $\tilde j_*(\sum_i\beta_i\otimes[\lambda_i])$, and the result follows from Lemma~\ref{oda}.

\end{proof}
It follows immediately from Levine's result (since $\tilde j=j\circ i$) that
\[
j_*\circ\tau_{k}\circ\rho:(\p_g)_k/(\p_g)_{k+1}\to H_1(\h) \otimes \bar\Lie_{k}
\]
is an embedding.

Let $F^{g-1}$ be the free group on $g-1$ generators.  There is a split exact sequence:
\begin{equation}\label{eq:pbraid}
1\to F^{g-1}\to\p_g\to\p_{g-1}\to 1.
\end{equation}
This exact sequence has a nice geometric description which we give now.  We think of $\p_g$ as the mapping class group of a disk $\Delta$ (relative its boundary) with $g$ marked points, each with a distinct marking.  The injection $F^{g-1}$ is given by thinking of $F^{g-1}$ as the fundamental group of $\Delta$ with $g-1$ punctures.  We then identify each of the punctures with the first $g-1$ marked points and identify the basepoint of the fundamental group with the $g^\textrm{th}$ puncture.  The injection $F^{g-1}=\pi_1(\Delta-(g-1\ \textrm{punctures}))\to\p_g$ is then given by the push map (see  \cite{primer}).
\begin{figure}[h!tp]
\centering
\includegraphics[width=.98\textwidth]{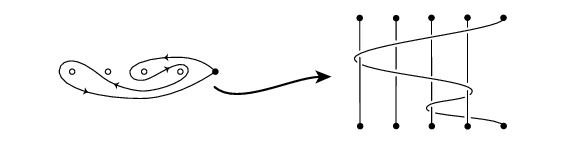}
\caption{The injection $F^{g-1}\to\p_g$.}
\end{figure}
The surjection $\p_g\to\p_{g-1}$ is given by forgetting the $g^\textrm{th}$ marked point.  
\begin{figure}[h!tp]
\centering
\includegraphics[width=.98\textwidth]{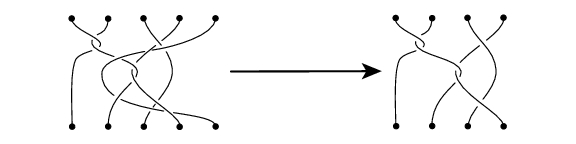}
\caption{The surjection $\p_g\to\p_{g-1}$.}
\end{figure}
It is clear that the sequence is exact.  There is also an easily understood splitting $\p_{g-1}\to\p_g$ given by simply appending a $g^\textrm{th}$ straight strand to any $(g-1)$-strand pure braid.
\begin{figure}[h!tp]
\centering
\includegraphics[width=.98\textwidth]{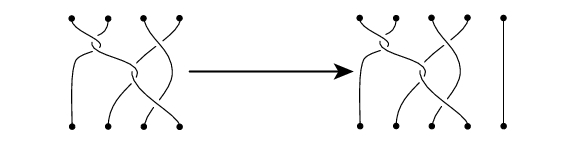}
\caption{The splitting $\p_{g-1}\to\p_{g}$.}
\end{figure}
We require the following theorem which is proved in \cite{FR}.
\begin{theorem}[Falk and Randell]\label{thm:fr}
Given a split short exact sequence
\[
1\to A\to B\to C\to1
\]
and regarding $A$ and $C$ as subgroups of $B$, if $[C,A]\subset[A,A]$ then there is a short exact sequence
\[
1\to\frac{A_k}{A_{k+1}}\to\frac{B_k}{B_{k+1}}\to\frac{C_k}{C_{k+1}}\to1
\]
for all $k$.
\end{theorem}
We wish to apply this theorem in the case of the sequence \eqref{eq:pbraid}.  We need only to check that $[p,f]\subset[F^{g-1},F^{g-1}]$, where $p\in\p_{g-1}$ and $f\in F^{g-1}$.  In fact it suffices to check this on a set of  generators for the groups $\p_{g-1}$ and $F^{g-1}$.  We appeal to Artin's presentation of the pure braid group \cite{artin}.  That is, the generators of $\p_g$ are denoted $s_{ij}$, where $i$ and $j$ range over the set $\{1,\cdots,g\}$, with $i<j$.  The element $s_{ij}$ ``wraps'' the $j^\textrm{th}$ strand around the $i^\textrm{th}$ strand as pictured in Figure~\ref{fig:artinpres}.
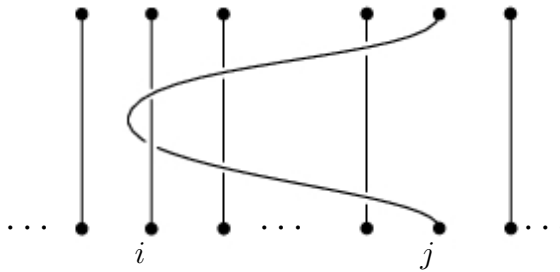
\begin{figure}[h!tp]
\centering
\begin{picture}(435,100)(0,0)
\includegraphics[width=.98\textwidth]{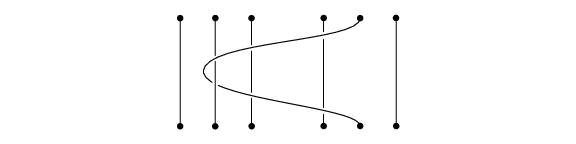}
\put(-230,10){\large$\cdots$}
\put(-326,10){\large$\cdots$}
\put(-130,10){\large$\cdots$}
\put(-277,0){$i$}
\put(-169,0){$j$}
\end{picture}
\caption{The element $s_{ij}$ from Artin's presentation of $\p_g$.}\label{fig:artinpres}
\end{figure}
The relations are as follows:
\begin{equation}\label{eq:artinrels}
s_{rs}^{-1}s_{ij}s_{rs}=\left\{
\begin{array}{cc}
s_{ij}&{\textrm{if}}\ r<s<i<j\\
s_{ij}&{\textrm{if}}\ i<r<s<j\\
s_{rj}s_{ij}s_{rj}^{-1}&{\textrm{if}}\ r<s=i<j\\
(s_{ij}s_{sj})s_{ij}(s_{ij}s_{sj})^{-1}&{\textrm{if}}\ r=i<s<j\\
\left[s_{rj},s_{sj}\right]s_{ij}\left[s_{rj},s_{sj}\right]^{-1}&{\textrm{if}}\ r<i<s<j
\end{array}
\right.
\end{equation}
Clearly $F^{g-1}$ is generated by $\{s_{ig}\}_{1\le i<g}$ and $\p_{g-1}$ is generated by $\{s_{ij}\}_{1\le i<j<g}$.  The desired result is obtained by replacing $j$ with $g$ in \eqref{eq:artinrels} and post-multiplying both sides of the equation by $s_{ig}^{-1}$.  Also, note that all possibilities are covered by the conditions listed at the right of \eqref{eq:artinrels}.  Therefore, Theorem~\ref{thm:fr} gives an exact sequence
\[
0
\to
\frac{(F^{g-1})_k}{(F^{g-1})_{k
+1}}
\to
\frac{(\p_g)_k}{(\p_g)_{k
+1}}
\to
\frac{(\p_{g-1})_k}{(\p_{g-1})_{k
+1}}
\to
0;
\]
thus, an isomorphism
\[
\frac{(\p_g)_k}{(\p_g)_{k
+1}}\cong\frac{(\p_{g-1})_k}{(\p_{g-1})_{k
+1}}\oplus\frac{(F^{g-1})_k}{(F^{g-1})_{k
+1}};
\]
hence,
\begin{equation}\label{eq:purebraidquotient}
\frac{(\p_g)_k}{(\p_g)_{k
+1}}\cong\bigoplus_{m=3}^g\frac{(F^{m-1})_k}{(F^{m-1})_{k
+1}}.
\end{equation}

\subsection{Proof of Theorem~\ref{robusthomeo}}\label{subsec:proof}

\begin{proof}[Proof of Theorem~\ref{robusthomeo}]
It only remains to prove that Inequality~\eqref{simpineq}:
\begin{equation*}\tag{5.2}
\frac{g(g+1)}{2}<\dim\im(j_*\circ\tau_k)^\R,
\end{equation*}
holds for large $g$.
The isomorphism~\eqref{eq:purebraidquotient} allows explicit computation of lower bounds for $\dim\im(j_*\circ\tau_k)^\R$.  In particular 
\begin{equation}\label{levinesestimate}
\sum_{m=3}^g\dim\left( \frac {(F^{m-1})_k} {(F^{m-1})_{k+1}}\right)^\R\le\dim\im(j_*\circ\tau_k)^\R.
\end{equation}
From Section~\ref{sec:freeLierings} we know that
\[
\dim\left( \frac {(F^{g-1})_k} {(F^{g-1})_{k+1}}\right)^\R
\]
is a polynomial in the variable $g$, of degree $k$, with positive leading coefficient.  It follows that
\begin{equation}\label{Fineq}
\frac{g(g+1)}{2}<\sum_{m=3}^g\dim\left( \frac {(F^{m-1})_k} {(F^{m-1})_{k+1}}\right)^\R\
\end{equation}
for sufficiently large $g$, provided that $k\ge3$, since the left-hand side is a polynomial of degree 2 and the right-hand side of degree $k$.  Combining Inequalities~\eqref{levinesestimate} and~\eqref{Fineq} yields Theorem~\ref{robusthomeo} for $k\ge3$.

In the case $k=2$, the Torelli case, $\im(j_*\circ\tau_k)$ is known exactly (see \cite{J80}).  It turns out that
\[
\dim\im(j_*\circ\tau_2)^\R={{g}\choose{3}},
\]
which is a polynomial of degree 3.  Thus, Inequality~\eqref{simpineq} holds in the case $k=2$, for large $g$. 
\end{proof}

\section{Genus considerations}
In this section we address the problem of how large the genus has to be for Inequality~\eqref{Fineq} to hold.  
From the results of Section~\ref{sec:freeLierings},
Inequality~\eqref{Fineq} becomes:
\begin{equation}\label{Cineq}
\frac{g(g+1)}{2}<\sum_{m=3}^gC(k,m-1).
\end{equation}
We wish to check Inequality~\eqref{Cineq} for some small number of ``boundary'' cases and deduce for which cases exactly it does and does not hold.  The following three propositions allow us to do so.
\begin{proposition}\label{kincreasing}
$C(k+1,m)\ge C(k,m)$, provided $k\ge2$ and $m\ge2$.  Hence, if Inequality~\eqref{Cineq} is satisfied for the pair $(k,g)$, then it is satisfied for the pair $(k+1,g)$, provided $k\ge2$ and $g\ge3$.
\end{proposition}
\begin{proof}
We define an injective function from the set of basic commutators of weight $k$ to the set of basic commutators of weight $k+1$ by $[w_1,w_2]\mapsto[[w_1,v],w_2]$, where $v$ is a basic commutator of weight 1, with $w_1>v$ and $w_2\ge v$.  This function is obviously injective since $[w_1,w_2]$ can be reconstructed from $[[w_1,v],w_2]$.  Note that, for a basic commutator of weight $k$ to have the form $[w_1,w_2]$ it is both necessary and sufficient that $k\ge2$ and $m\ge2$.
\end{proof}

\begin{proposition}\label{gincreasing}
If Inequality~\eqref{Cineq} is satisfied for a pair $(k,g)$, with $k=3$ and $g\ge3$, then it is also satisfied for the pair $(k,g+1)$. 
\end{proposition}
\begin{proof}
By assumption we have:
\[
\frac{g(g+1)}{2}<\sum_{m=3}^gC(3,m-1).
\]
If we can show that
\begin{equation}\label{deltaC}
\frac{(g+1)(g+2)}{2}-\frac{g(g+1)}{2}\le\sum_{m=3}^{g+1}C(3,m-1)-\sum_{m=3}^gC(3,m-1),
\end{equation}
then we are done.  The left hand side of Inequality~\eqref{deltaC} is equal to:
\[
\frac{g+1}{2}(g+2-g)=g+1.
\]
The right hand side is $C(3,g)=(g^3-g)/3$, by the Witt formula of Theorem~\ref{witt}.  It is easy to see (for instance by taking derivatives) that $g+1\le (g^3-g)/3$, provided that $g\ge3$.
\end{proof}

Finally, in the case $k=2$, the Torelli case, we need the following.
\begin{proposition}\label{prop:tincreasing}
If $k=2$, then Inequality~\eqref{ineq} holds provided $g\ge7$.
\end{proposition}
\begin{proof}
As mentioned at the end of Section~\ref{robusthomeos}, it is known that
\[
\dim\im(j_*\circ\tau_2)^\R={{g}\choose{3}}.
\]
Inequality~\eqref{ineq} then becomes:
\begin{equation}\label{twoineq}
\frac{g(g+1)}{2}<{{g}\choose{3}},
\end{equation}
which holds for $g\ge7$, but not for $g<7$.
\end{proof}
For the first several values of $k$ we record the polynomial $C(k,m-1)$:
\begin{align*}
C(3,m-1)=&\frac{(m-1)^3-(m-1)}{3}\\
C(4,m-1)=&\frac{(m-1)^4-(m-1)^2}{4}\\
C(5,m-1)=&\frac{(m-1)^5-(m-1)}{5}\\
C(6,m-1)=&\frac{(m-1)^6-(m-1)^3-(m-1)^2+(m-1)}{6}\\
C(7,m-1)=&\frac{(m-1)^7-(m-1)}{7}\\
C(8,m-1)=&\frac{(m-1)^8-(m-1)^4}{8}
\end{align*}
Using these formulas we calculate the values of the right-hand side of \eqref{Cineq}, for small $k$ and $g$, in Figure~\ref{RHSvals}.  Of course, for $k=2$ we use ${{g}\choose{3}}$. 
\begin{figure}[h!tp]
\centering
\quad\ \subfigure[Values of the right-hand side of Inequality~\eqref{Cineq}]
{\begin{tabular}{c||*{7}{c|}}
\cline{1-5}
$k=8$&0&30&840&9000\\
\cline{1-5}
$k=7$&0&18&330&2670\\
\cline{1-6}
$k=6$&0&9&125&795&3375\\
\cline{1-8}
$k=5$&0&6&54&258&882&2436&5796\\
\cline{1-8}
$k=4$&0&3&21&81&231&546&1134\\
\cline{1-8}
$k=3$&0&2&10&30&70&140&252\\
\cline{1-8}
$k=2$&0&1&4&10&20&35&56\\
\hline\hline
&g=2&g=3&g=4&g=5&g=6&g=7&g=8
\end{tabular}}\\
\subfigure[Values of $g(g+1)/2$]
{\begin{tabular}{c||*{7}{c|}}
\cline{1-8}
$g(g+1)/2$&3&6&10&15&21&28&36\\
\hline\hline
&g=2&g=3&g=4&g=5&g=6&g=7&g=8
\end{tabular}}\quad\ 
\caption{Values of the two sides of Inequalities~\eqref{Cineq} and \eqref{twoineq}}\label{RHSvals}
\end{figure}
From Figure~\ref{RHSvals} we can easily see when Inequalities~\eqref{Cineq} and \eqref{twoineq} do and do not hold.  We record this data in Figure~\ref{binary}.  The darkly shaded boxes signify that the appropriate inequality does hold and has been checked explicitly (by observing Figure~\ref{RHSvals}).  Vertical arrows indicate cases where the inequality holds due to Proposition~\ref{kincreasing} and horizontal arrows indicate where it holds due to Propositions~\ref{gincreasing} or \ref{prop:tincreasing}.  Squares left blank are where the inequality does not hold.
\begin{figure}[h!tp]
\centering
\begin{picture}(435,400)(0,0)
\includegraphics[width=.98\textwidth]{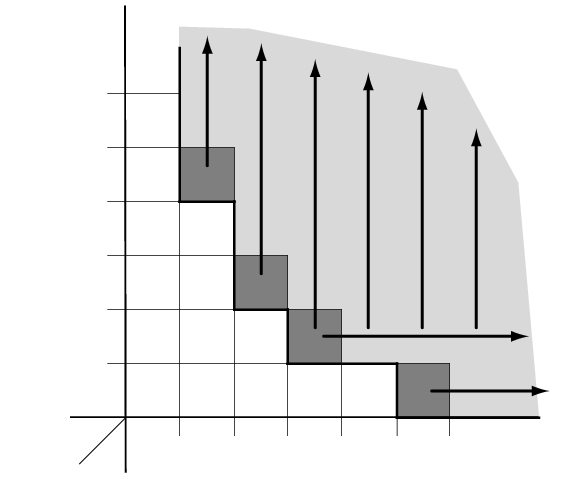}
\put(-155,113){Prop.~\ref{gincreasing}}
\put(-90,90){Prop.~\ref{prop:tincreasing}}
\put(-232,220){\rotatebox{90}{Prop.~\ref{kincreasing}}}
\put(-65,160){\Large$\cdots$}
\put(-375,46){\Large$k$}
\put(-375,80){\Large2}
\put(-375,120){\Large3}
\put(-375,161){\Large4}
\put(-375,202){\Large5}
\put(-375,242){\Large6}
\put(-375,283){\Large7}
\put(-355,30){\Large$g$}
\put(-324,30){\Large2}
\put(-283,30){\Large3}
\put(-242,30){\Large4}
\put(-201,30){\Large5}
\put(-160,30){\Large6}
\put(-119,30){\Large7}
\end{picture}
\caption{Region where Inequalities~\eqref{Cineq} and \eqref{twoineq} hold.}
\label{binary}
\end{figure}
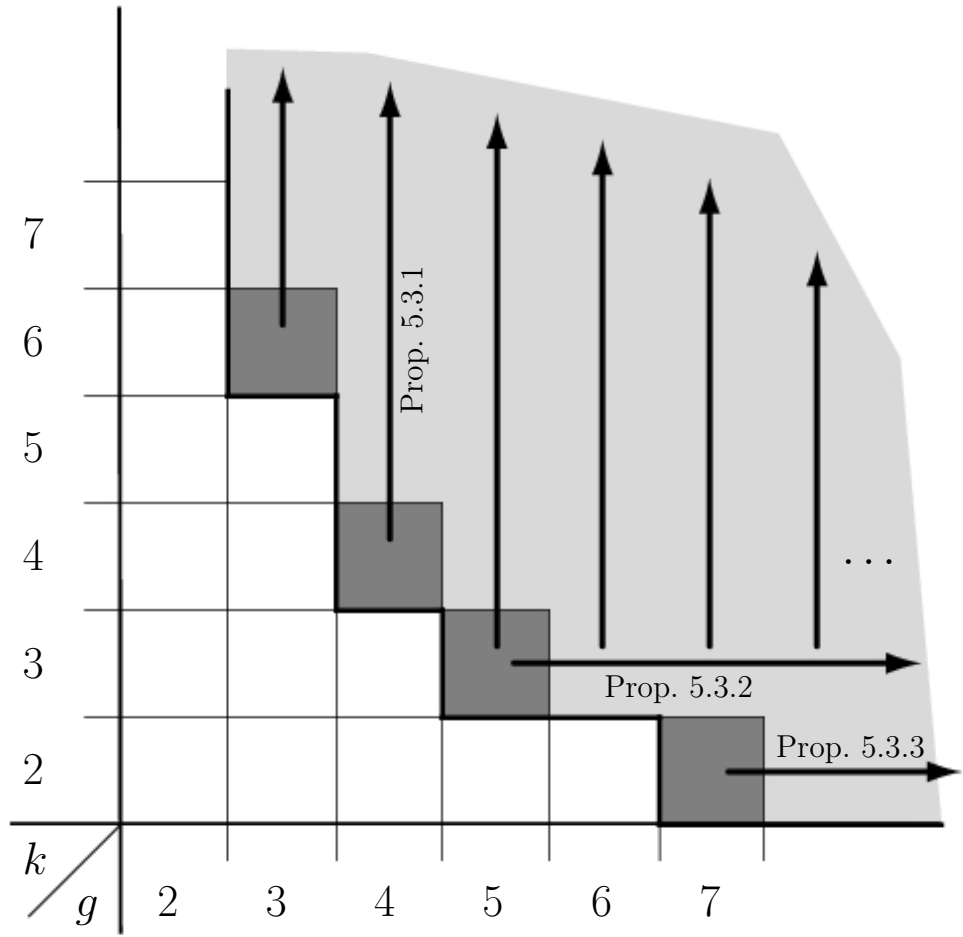

\end{chapter}

\begin{chapter}{The main theorem}\label{sec:main}

\begin{theorem}\label{main}
Given an integer $k\ge2$, there are homeomorphisms $[f]\in\J(k)-\J(k+1)$ that do not extend to any handlebody, provided that the ordered pair $(g,k)$ lies in the shaded region shown in Figure~\ref{binary}.
  
Furthermore, $f^n$ does not extend to any handlebody, for all integers $n\ne0$.
\end{theorem}

It follows that:

\begin{corollary}\label{mc}
If $\Sigma$ is a closed, orientable surface of genus 3 or greater, then there are homeomorphisms of $\Sigma$ that do not extend to any handlebody, lying arbitrarily deep in the Johnson filtration.
\end{corollary}

\begin{proof}
By corollary~\ref{robust}, we need only to prove the existence of robust homeomorphisms in $\J(k)-\J(k+1)$.  Since $\J(k+1)$ is the kernel of $\tau_{1,k}$, any robust homeomorphism is clearly not in $\J(k+1)$.  Therefore, Theorem~\ref{robusthomeo} and the computation of the previous section implies the first statement of the theorem.  For the second part we recall that a homeomorphism $f$ is robust if and only if $f^n$ is robust, for $n\ne0$.
\end{proof}

Let us summarize what is known about homeomorphisms that do not extend to any handlebody in relation to the Johnson filtration.  As mentioned earlier, Hain proved the content of Theorem~\ref{main} by different methods. Hain also showed that the result holds for the cases $(g=3,k=5)$ and $(g=4,k=3)$.  Next we turn our attention to the result attributed to Casson and Johannson-Johnson.  For genus 2 their method finds homeomorphisms in $\J(3)-\J(4)$.  For larger genera their method finds homeomorphisms in $\J(2)-\J(4)$, but it is not clear whether these homeomorphisms lie in $\J(3)$ or not.  

The question of whether there are homeomorphisms lying arbitrarily deep in the Johnson filtration for genus 2 appears to be completely open.
\end{chapter}

\end{spacing}

    \bibliographystyle{amsalpha}
    \bibliography{johnsonbib}

\end{document}